\documentclass[11pt,reqno]{amsart}
\usepackage{preprintStyle}

\newcommand{\N}{\ensuremath\mathbb{N}}

\newcommand{\R}{\ensuremath\mathbb{R}}

\newcommand{\T}{\ensuremath\mathsf{T}}

\newcommand{\ds}{\,\mathrm{d}s}

\DeclareMathOperator{\real}{Re}

\newcommand{\calI}{\mathcal{I}}

\newcommand{\calK}{\mathcal{K}}

\newcommand{\Lone}{L^1}
\newcommand{\Ltwo}{L^2}
\newcommand{\Lp}{L^p}

\definecolor{mycolor1}{rgb}{0.00000,0.44700,0.74100}
\definecolor{mycolor2}{rgb}{0.85000,0.32500,0.09800}
\definecolor{mycolor3}{rgb}{0.92900,0.69400,0.12500}
\definecolor{mycolor4}{rgb}{0.46600,0.67400,0.18800}
\definecolor{mycolor5}{rgb}{0.49400,0.18400,0.55600}

\definecolor{nicegreen}{rgb}{0.3,0.7,0.4}

\newcommand{\timeInt}{\mathbb{T}}

\newcommand{\state}{x}

\renewcommand{\i}{\mathrm{i}}
\DeclareMathOperator{\Real}{Re}

\newcommand{\linop}[2]{\mathcal{L}(#1, #2)}
\newcommand{\linopo}[1]{\mathcal{L}(#1)}

\newcommand{\mexp}[1]{\mathrm{e}^{#1}}
\newcommand{\residualbc}{\delta_{\mathrm{b}}}
\newcommand{\residualbcdot}{\dot{\delta}_{\mathrm{b}}}
\newcommand{\residualinit}{\delta_{\mathrm{0}}}
\newcommand{\residualevo}{\delta}

\newcommand{\coordnk}[1]{x_{n, #1}}
\newcommand{\deltacoordn}{\Delta_n}

\newcommand{\statespace}{Z}
\newcommand{\statespacewnorm}{(\statespace, \|\cdot \|_{\statespace})}
\renewcommand{\state}{z}
\newcommand{\statebc}{\state_{\mathrm{b}}}
\newcommand{\stateinit}{\state_{\mathrm{0}}}

\newcommand{\stateapprox}{\tilde{\state}}

\newcommand{\approxstatespace}{\statespace_n}

\newcommand{\staten}{\state_n}
\newcommand{\statendot}{\dot{\state}_n}
\newcommand{\statenk}{\state_{n, k}}

\newcommand{\stategeneral}{\state} 
\newcommand{\stategeneralbc}{\state_{\mathrm{b}}} 
\newcommand{\stategeneralbcdot}{\dot{\state}_{\mathrm{b}}} 
\newcommand{\stategeneralinit}{\state_{\mathrm{0}}} 

\newcommand{\stategeneralapprox}{\overline{\state}_n} 

\newcommand{\statefattorini}{\xi}

\newcommand{\stateerr}{\mathrm{e}}

\newcommand{\controlspace}{U}
\newcommand{\controlspacewnorm}{(\controlspace, \|\cdot \|_{\controlspace})}

\newcommand{\vel}{\mathbf{v}}

\newcommand{\bop}{\mathfrak{D}}
\newcommand{\boprinv}{\mathfrak{D}_0}
\newcommand{\bopn}{\mathfrak{D}_n}
\newcommand{\bopnrinv}{\mathfrak{D}_{n,0}}
\newcommand{\A}{\mathfrak{A}}
\newcommand{\Agen}{\mathcal{A}}
\newcommand{\An}{\mathfrak{A}_n}
\newcommand{\Angen}{\mathcal{A}_n}
\renewcommand{\S}{\mathcal{S}}
\newcommand{\Sn}{\mathcal{S}_n}
\newcommand{\En}{\mathcal{E}_n}
\newcommand{\Pn}{\mathcal{P}_n}

\newcommand{\mup}{\mu_{\mathrm{p}}}
\newcommand{\mue}{\mu_{\mathrm{e}}}

\renewcommand{\d}[1]{\,\mathrm{d}#1}

\newcommand{\narrowinfty}{\xrightarrow[]{n\rightarrow\infty}}
\newcommand{\narrowinftySOT}{\xrightarrow[\mathrm{SOT}]{n\rightarrow\infty}}

\newcommand{\abbr}[1]{\textsf{#1}\xspace}
\newcommand{\FE}{\abbr{FE}}
\newcommand{\FEM}{\abbr{FEM}}
\newcommand{\BIVP}{\abbr{BIVP}}
\newcommand{\BIVPs}{\abbr{BIVPs}}

\newcommand{\NN}{\abbr{NN}}
\newcommand{\PINN}{\abbr{PINN}}
\newcommand{\PINNs}{\abbr{PINNs}}
\newcommand{\vPINNs}{\abbr{vPINNs}}
\newcommand{\LPINNs}{\abbr{LPINNs}}
\newcommand{\XPINNs}{\abbr{XPINNs}}
\newcommand{\ISS}{\abbr{ISS}}
\newcommand{\PDE}{\abbr{PDE}}
\newcommand{\PDEs}{\abbr{PDEs}}

\crefname{assumption}{Assumption}{Assumptions} 

\newcommand{\plotlinewidth}{1.5pt}

\definecolor{c_rose}{HTML}{CC6677}
\definecolor{c_indigo}{HTML}{332288}
\definecolor{c_sand}{HTML}{DDCC77}
\definecolor{c_green}{HTML}{117733}
\definecolor{c_cyan}{HTML}{88CCEE}
\definecolor{c_wine}{HTML}{882255}
\definecolor{c_teal}{HTML}{44AA99}
\definecolor{c_olive}{HTML}{999933}
\definecolor{c_purple}{HTML}{AA4499}

\colorlet{c_init}{c_olive}
\colorlet{c_evo}{c_rose}
\colorlet{c_int_1}{c_sand}
\colorlet{c_int_2}{c_olive}
\colorlet{c_bt}{c_cyan}
\colorlet{c_bzero}{c_teal}
\colorlet{c_tot}{c_wine}
\colorlet{c_ref}{c_green}

\title[Prediction error certification for PINNs]{Prediction error certification for PINNs: Theory, computation, and application to Stokes flow}

\author{Birgit Hillebrecht${}^\star$ \and Benjamin Unger${}^\dagger$}

\address{${}^{\star}$ ORCID:0000-0001-5361-0505, Institute for Applied and Numerical Mathematics, Karlsruhe Institute of Technology, 76131 Karlsruhe, Germany}
\email{birgit.hillebrecht@kit.edu}

\address{${}^{\dagger}$ ORCID: 0000-0003-4272-1079, Institute for Applied and Numerical Mathematics, Karlsruhe Institute of Technology, 76131 Karlsruhe, Germany}
\email{benjamin.unger@kit.edu}

\begin{document}

\begin{abstract}
	Rigorous error estimation is a fundamental topic in numerical analysis. With the increasing use of physics-informed neural networks (PINNs) for solving partial differential equations, several approaches have been developed to quantify the associated prediction error. In this work, we build upon a semigroup-based framework previously introduced by the authors for estimating the PINN error. While this estimator has so far been limited to academic examples -- due to the need to compute quantities related to input-to-state stability -- we extend its applicability to a significantly broader class of problems. This is accomplished by modifying the error bound and proposing numerical strategies to approximate the required stability parameters. The extended framework enables the certification of PINN predictions in more realistic scenarios, as demonstrated by a numerical study of Stokes flow around a cylinder.
\end{abstract}

\maketitle

\smallskip




\section{Introduction}
\label{sec::Introduction}

Scientific machine learning, which refers to the use of machine learning methods to learn solution maps of dynamical systems for predictive purposes, has received considerable attention in recent years. Although the practical benefits of neural networks over classical methods for dynamical systems have yet to be shown or are even challenged \cite{GroKLS24}, machine learning still offers attractive prospects, such as the potential to overcome the curse of dimensionality.  

Since the introduction of \emph{physics-informed neural networks} (\PINNs) for approximating solutions to partial differential equations (\PDEs) \cite{RaiPK19}, research has rapidly expanded to encompass various adaptations, including variational formulations (\vPINNs) \cite{RadKA23,KhaZK19}, Lagrangian formulations (\LPINNs) \cite{MojBH23}, and domain decomposition strategies (\XPINNs) \cite{JagK20}. The increasing interest from both methodological and application-driven perspectives (e.g., \cite{CaiMWYK21,RadKA23, SchBGEM24, HagRMGJ21}) has led to comprehensive survey articles such as \cite{KarKLPWY21, CuoDGRRP22, MhaTJ25}.

While paving the way for more competitiveness with traditional numerical methods, considerable effort has been devoted to establishing convergence guarantees and rigorous error analysis for scientific machine learning. \emph{A priori} error estimates have been derived for either specific problems or specific network geometries \cite{DeRM22, DeRJM24, DeRM24, RojMMPP24, BerCP22_2} which also show convergence properties of the \PINN as numerical algorithm. While these results ensure the existence of neural networks that approximate the exact solution within a prescribed tolerance, they offer no insight into the prediction quality of a specific trained model. In fact, \cite{GroV24} highlights that the theory-to-practice gap in deep learning prevents algorithms based on finite sampled data from achieving the theoretical optimum. Consequently, one must resort to \emph{a posteriori} error estimates to assess the accuracy of a given trained \PINN. In this context, \cite{HilU22,HilU25} derive a rigorous error bound based on semigroup theory, \cite{BerCP22} develops a computable a posteriori error estimator for \vPINNs, and \cite{ErnU24} proposes a wavelet-based estimator. 

In this paper, we focus on the practical applicability of the a posteriori error estimators of \cite{HilU25} by developing tools to numerically estimate the required constants in the error certificates for challenging engineering and applied sciences problems. In more detail, given Banach spaces $\statespacewnorm$ and $\controlspacewnorm$, a time interval $\timeInt \vcentcolon= [0,T]\subseteq\R$,  and a spatial domain $\Omega\subseteq\R^{d}$, we consider the \emph{boundary and initial value problem} (\BIVP) 
\begin{equation}
    \label{eq::BIVP}
    \left\{\quad 
    \begin{aligned}
        \dot{\state} & = \A \state \quad & \mathrm{in}\; \timeInt \times \Omega ,\\
        \bop \state &= \statebc \quad & \mathrm{on} \;\timeInt \times \partial \Omega ,\\ 
        \state &= \state_0 \quad & \mathrm{in}\; \left\lbrace t=0 \right\rbrace \times \Omega,
    \end{aligned}\right.
\end{equation}
where $\A\colon D(\A)\subseteq \statespace\to \statespace$ is a linear (potentially unbounded) operator, the boundary operator $\bop \colon D(\bop)\subseteq \statespace \rightarrow U$ with $D(\A) \subseteq D(\bop)$ is a linear operator, and $\state_{\mathrm{b}}\in L^1(\timeInt; U)$ and $x_0\in \statespace$ are the boundary and initial values, respectively. It is worth noting, that the dynamical systems under consideration are assumed to be \PDEs. In particular, the differential operator $\A$ might contain derivatives w.r.t.~the spatial coordinates. For example for the heat equation 
\begin{equation*}
    \begin{aligned}
        \partial_t \state (t, x) = \alpha \partial_{xx} \state(t, x)
    \end{aligned}
\end{equation*}
the operator is given by $\A = \alpha \partial_{xx}$. We refer to the forthcoming \Cref{subsec:heatEquation} for further details on the associated spaces and the boundary operator.

In our setting, training a \PINN to approximate~\eqref{eq::BIVP} means that we design a candidate function $\stateapprox$, typically a deep neural network, that takes the time $t$ (and maybe further model parameters) as input and compute an approximation $\stateapprox(t)\in \statespace$ of $\state(t)\in \statespace$. Assuming that~$\stateapprox$ is sufficiently smooth, the \PINN approximation satisfies the perturbed \BIVP
\begin{equation}
    \label{eq::BIVP_Approx}
    \left\{\quad
    \begin{aligned}
        \dot{\stateapprox} & = \A \stateapprox + \residualevo \quad & \mathrm{in} \;\timeInt \times \Omega, \\
        \bop \stateapprox &= \statebc + \residualbc \quad & \mathrm{in}\; \timeInt \times \partial \Omega, \\ 
        \stateapprox &= x_0 + \residualinit \quad & \mathrm{in} \;\left\lbrace t=0 \right\rbrace \times \Omega, 
    \end{aligned}\right.
\end{equation}
with error contributions $\residualevo \colon \timeInt \rightarrow \statespace $, $\residualbc \colon \timeInt \rightarrow U$, and $\residualinit \in \statespace$ for the state equation, the boundary condition, and the initial value, respectively. For instance, the approximation is in general smooth enough if the hyperbolic tangent is used as an activation function. We emphasize that for a given \PINN instance the error contributions $\residualevo, \residualbc, \residualinit$ can be computed efficiently using algorithmic differentiation to calculate the residuals of $\stateapprox$ in the \BIVP~\eqref{eq::BIVP}. 

When deriving an error bound or error estimator for~$\|\state(t)-\stateapprox(t)\|_{\statespace}$, it is important to note that the contribution of the boundary error might be significant. Using \emph{input-to-state stability} (\ISS) \cite{Son89}, semigroup theory \cite{Paz89}, the forthcoming \Cref{ass::well_posed}, and the set of growth functions 
\begin{equation*}
	\calK \vcentcolon= \{\mu\colon \R_{0}^+\to\R_0^+ \mid \mu(0) = 0, \mu \text{ continuous, strictly increasing}\},
\end{equation*}
the following a posteriori error bound was derived in \cite{HilU25}.
\begin{theorem}[{\!\!\cite[Thm.~2]{HilU25}}]
	\label{thm:errorBound}
	Let the \BIVPs~\eqref{eq::BIVP} and~\eqref{eq::BIVP_Approx} satisfy \Cref{ass::well_posed}. Then there exists constants $M>0$, $\omega\in\R$ and a growth function $\gamma\in\calK$ such that the mild solutions $\state$ and $\stateapprox$ of~\eqref{eq::BIVP} and \eqref{eq::BIVP_Approx} satisfy
	\begin{equation}
		\label{eq::errorBound}
    	\|\state(t)-\stateapprox(t)\|_{\statespace} \le M \mexp{\omega t} \|\delta_0 \|_{\statespace} + \int_0^t M\mexp{\omega (t-s)}\|\delta(s)\|_{\statespace} \ds + \gamma (\|\residualbc\|_{L_\infty(0,t; U)}).
	\end{equation}
\end{theorem}

To apply \Cref{thm:errorBound} to practical problems, it is necessary to determine the constants $M$ and $\omega$, as well as the growth function $\gamma$. For simple benchmark problems, such as the heat equation on a bounded one-dimensional domain, these quantities can be computed analytically; see \cite[Sec.,V.A]{HilU25}, which builds on results from \cite{SeiTW22} and \cite{JacNPS18}.
For more complex problems and geometrically challenging domains, however, a direct computation of the \abbr{ISS} functions remain infeasible. The goal of this work is to partially bridge this gap, thereby enabling the application of \Cref{thm:errorBound} to a broader class of problems. Specifically, our main contributions are as follows:
\begin{enumerate}
	\item We show in \Cref{prop::approx_growth_bound} that the constants $M$ and $\omega$ in~\eqref{eq::errorBound} can be approximated via established numerical approximation methods that do not require solving the \BIVP~\eqref{eq::BIVP}. 
	\item Instead of approximating the growth function $\gamma$ in~\eqref{eq::errorBound} numerically, we use the Fattorini trick \cite{Fat68} to derive a modified error bound in \Cref{cor::error_estimator}, which circumvents the need for a suitable growth function $\gamma\in\calK$ to be available. Together with \Cref{prop::approx_growth_bound}, this enables us to compute the error bound for a large class of problems. 
	\item For a one-dimensional heat equation, we prove in \Cref{lem:heatEquation} that the approximated quantities agree with the analytically derived constants from the literature and show that with additional knowledge, the error bound can be further improved; cf.~\Cref{cor:heatEquationErrBound}. We then train a \PINN to approximate the solution of the heat equation and observe an excellent agreement of the error bound with the true error; see \Cref{fig::heat_error2} for details.
	\item We demonstrate the applicability of our error estimator to a more challenging problem by investigating a two-dimensional Stokes flow around a cylindrical obstacle in \Cref{subsec:stokes}. To improve the accuracy of the \PINN approximation, we use harmonic feature embeddings from \cite{KasH24} in a modified fashion, which enables time-dependent Dirichlet boundary conditions. The numerical results confirm that our error estimator provides a rigorous upper bound on the actual prediction error.
\end{enumerate}

We emphasize that the homogenization approach to boundary perturbations, which we employ in the modified error bound, is also used in \cite{GaoZ24}. In contrast to \cite{GaoZ24}, however, our proposed error estimator avoids the use of the Gronwall lemma.

\subsection*{Notation}
The space $\Lp(\timeInt; U)$ denotes the space of p-Bochner-integrable functions from $\timeInt$ to $U$. Further, the space of linear bounded operators from a Banach space $\statespace$ to a Banach space $Y$ is denoted by $\linop{\statespace}{Y}$. If $\statespace=Y$, the we simply write $\linopo{\statespace}$. We use common notation for Sobolev spaces on $\Omega \subseteq \R^n$ and for functions from $I$ to~$\statespace$ denoted by $W^{k, p}(\Omega)$ and $W^{k, p}(I, \statespace)$, respectively. The space of k-times continuously differentiable functions with compact support on $Z \subseteq \R^n$ is denoted by $C^k_c(Z)$.


\section{Prerequisites}
\label{sec::Prerequisites}

For the error bounds in \Cref{thm:errorBound} and the forthcoming \Cref{cor::error_estimator}, existence and uniqueness of solutions of the \BIVPs~\eqref{eq::BIVP} and~\eqref{eq::BIVP_Approx} is presumed. This assumption is justified and elucidated in \Cref{subsec:wellPosed}. We further recall the assumptions for the Trotter-Kato approximation theorem in \Cref{subsec:approxSequence} to reuse these in the approximation results derived in later sections.

\subsection{Well-posedness of the BIVP}
\label{subsec:wellPosed}
Throughout the manuscript, we work with a mild solution concept and follow closely the presentation from \cite[Cha.~10]{CurZ20}. Consider the general \BIVP (as common abstraction for~\eqref{eq::BIVP} and~\eqref{eq::BIVP_Approx})
\begin{equation}
    \label{eq::BIVP:gen}
    \left\{\quad
    \begin{aligned}
        \dot{\stategeneral} & = \A \stategeneral + f \quad & \mathrm{in} \;\timeInt \times \Omega, \\
        \bop \stategeneral &= \stategeneralbc \quad & \mathrm{in}\; \timeInt \times \partial \Omega, \\ 
        \stategeneral &= \stategeneralinit \quad & \mathrm{in} \;\left\lbrace t=0 \right\rbrace \times \Omega.
    \end{aligned}\right. 
\end{equation}

\begin{assumption} 
    \label{ass::well_posed}
    The operator $\Agen \coloneqq \A |_{\ker\bop}$ is the generator of a strongly continuous semigroup~$\S(t)$ on $\statespace$ and $\bop$ has a linear bounded right-inverse $\boprinv \in \linop{\controlspace}{(D(\A), \| \cdot \|_\A)}$.
\end{assumption}

\begin{remark}
    The existence of a linear or non-linear right-inverse of a boundary trace operator has been and is researched extensively. Exemplarily, the works \cite{Gag57, Gre87} give results on conditions when a linear bounded right-inverse of a trace operator exists, as opposed to results stating that there can not exist a linear right-inverse such as in \cite{Pee79, MalSS15}. These results emphasize that an informed choice of both the boundary condition space $\controlspace$ and the solution space $\statespace$ is required to apply the derived results. 
\end{remark}

\begin{example}
	Consider the half-plane $\Omega = \R^n \times (0, \infty )\subseteq \R^{n+1}$ with $n$-dimensional flat boundary, i.e., $\R^n \times \{0\}$. To derive a linear right-inverse for the corresponding Dirichlet boundary trace operator, we introduce the function spaces $\statespace=W^{1,p}(\Omega)$ and $\controlspace = C_c^\infty(\R^n) \backslash \lbrace 0 \rbrace$ (c.f. \cite[Thm. 18.13]{Leo17}). We construct a right inverse of the Dirichlet trace operator based on a $\zeta \in C_c^\infty([0, \infty))$ fulfilling $\zeta(0) = 1$ as   
    \begin{equation*}
        ( \boprinv^{\zeta, \delta} g )(x, x_{n+1})  \coloneqq  g(x)\zeta(x_{n+1}/\delta)
    \end{equation*}
    for $g \in \controlspace$. The construction requires the function $\zeta$ with compact support, i.e., there is a $\hat{x} \in (0, \infty)$ such that $\zeta(\hat{x}) = 0$. This function smoothly decays the value of the boundary into the domain until it reaches 0 at latest at $x_{n+1} = \hat{x}/\delta$. Further, if fulfills by construction $tr ( \boprinv^{\zeta, \delta} g ) = g$.
\end{example}

If $f$ and $\stategeneralbc$ are sufficiently smooth, then the classical solution of~\eqref{eq::BIVP:gen} can be transformed to $\statefattorini \coloneqq \stategeneral - \boprinv \stategeneralbc$ such that the transformed \BIVP reads
\begin{equation}
    \label{eq::BIVP:gen:Fattorini}
    \left\{\quad
    \begin{aligned}
        \dot{\statefattorini} & = \Agen \statefattorini + f - \boprinv \stategeneralbcdot + \A \boprinv \stategeneralbc \quad & \mathrm{in}\; \timeInt \times \Omega ,\\
        \bop z &= 0 \quad & \mathrm{in}\; \timeInt \times \partial\Omega ,\\
        \statefattorini &= \stategeneralinit - \boprinv \stategeneralbc(0) \quad & \mathrm{in}\; \left\lbrace t=0 \right\rbrace \times \Omega.
    \end{aligned}\right.
\end{equation}
Rewriting the \BIVP~\eqref{eq::BIVP:gen} in the form~\eqref{eq::BIVP:gen:Fattorini} is due to \cite{Fat68}, and hence referred to as the Fattorini trick in the literature. 
The transformed \BIVP~\eqref{eq::BIVP:gen:Fattorini} motivates the following definition.

\begin{definition}[Solution concept]
	\label{def:mildSol}
	A map $\stategeneral\colon \timeInt\to \statespace$ is called a \emph{mild solution} for~\eqref{eq::BIVP:gen} if 
	\begin{align*}
		\state(t) = \boprinv \stategeneralbc(t) + \S(t)\big(\stategeneralinit - \boprinv \stategeneralbc(0)\big) + \int_0^t \S(t-s)\big( f(s) - \boprinv \stategeneralbcdot(s) + \A\boprinv \stategeneralbc(s)\big)\ds.
	\end{align*}
\end{definition}

Clearly, every classical solution of~\eqref{eq::BIVP:gen} is also a mild solution. The converse direction is in general not true. Nevertheless, we obtain the following well-posedness result for mild solutions.

\begin{proposition}[{\!\!\cite[Lem.~5.1.5]{CurZ20}}]
	\label{prop:wellPosed}
	Consider the \BIVP~\eqref{eq::BIVP:gen}, let \Cref{ass::well_posed} be satisfied, and assume $\stategeneralinit - \boprinv \stategeneralbc(0) \in D(\Agen)$, $f\in \Lp(\timeInt,\statespace)$, and $\stategeneralbc\in W^{1,p}(\timeInt,\statespace)$ for some $p\geq 1$. Then there exists a unique mild solution of~\eqref{eq::BIVP:gen} and the mild solution is continuous in $\timeInt$.
\end{proposition}

\begin{remark}
	One can also define a solution without requiring $\stategeneralbc\in W^{1,p}(\timeInt,\statespace)$; see for instance \cite{Sch20} and the references therein. Following this approach will eventually yield the error bound from \Cref{thm:errorBound} presented in \cite{HilU25}. However, in many practical examples, the boundary function $\statebc$ is smooth. If we further assume the \PINN activation function to be smooth, then also the \PINN boundary error $\residualbc$ is sufficiently smooth. 
\end{remark}

\subsection{Approximation sequences for semigroups}
\label{subsec:approxSequence}
In the following, we use the general approximation setting for the Trotter-Kato theorem in different Banach spaces as proposed in \cite{ItoK98} with notation similar to the one introduced in \cite[Def.~1.1]{Nam23}.

\begin{definition}[Approximation sequence of Banach spaces]
    \label{def::approx_sequence}
    Consider a sequence of triples $(\approxstatespace, \Pn, \En)_{n\in \N}$, each consisting of a Banach space $(\approxstatespace,\|\cdot\|_{\approxstatespace})$, a bounded linear operator $\Pn \colon \statespace \to \approxstatespace$, and a right-inverse to $\Pn$ denoted by $\En \colon \approxstatespace \to \statespace$. We call $(\approxstatespace, \Pn, \En)_{n\in \N}$ an \emph{approximation sequence} for $\statespace$ if
    \begin{enumerate}
    	\item $\lim_{n\rightarrow \infty} \| \Pn \state\|_{\approxstatespace} = \| \state\|_{\statespace}$ for every $\state\in \statespace$ (c.f. \cite[Def. 1.1]{Nam23}), and
    	\item there exist constants $\mup, \mue>0$ such that $\|\Pn\|_{\linop{\statespace}{\approxstatespace}} \le \mup$ and $\|\En \|_{\linop{\approxstatespace}{\statespace}} \le \mue$ for all $n\in\N$.
    \end{enumerate}
\end{definition}

Moreover, we say that $\staten \in \approxstatespace$ converges against $\state \in \statespace$ if  $\|\Pn \state - \staten \|_{\approxstatespace} \narrowinfty 0$; see \cite{Kur70,Nam23}. In this case, we write $\lim_{n\rightarrow\infty} \staten = \state$.

Given an approximation sequence $(\approxstatespace, \Pn, \En)_{n \in \N}$, a semigroup of operators $\S(t)$ on $\statespace$, and a sequence of semigroups of operators $(\Sn(t))_{n\in\N}$ on $\approxstatespace$. Then $(\Sn(t))_{n\in\N}$ is called \emph{convergent} to $\S(t)$ if 
\begin{equation*}
	\lim_{n\rightarrow\infty} \En\Sn(t) \staten = \S(t) \state \quad \text{for all}\;\; t\ge0,
\end{equation*}
for all $\staten\in\approxstatespace$ and $\state\in \statespace$ with $\lim_{n\to\infty} \staten = \state$. We denote convergent semigroups in this sense by $\Sn(t) \narrowinftySOT \S(t)$.


\section{Error estimation and numerical approximation of operator norms and growth bounds}
\label{sec::theory}

With these preparations, we can now state the main theoretical results of this paper, namely a novel a posteriori error bound for \PINNs and a computational method to compute the required constants. We start with the novel error bound by leveraging the mild solution concept from \Cref{def:mildSol} and the underlying Fattorini trick.

\begin{theorem}
    \label{cor::error_estimator}
    Consider the \BIVP~\eqref{eq::BIVP} and its approximation~\eqref{eq::BIVP_Approx}, let \Cref{ass::well_posed} be satisfied, and assume $\stategeneralbc, \residualbc \in W^{1,p}(\timeInt; \statespace)$ and $\residualevo\in \Lp(\timeInt; \statespace)$. Let $\state$ and $\stateapprox$ be the mild solutions of~\eqref{eq::BIVP} and \eqref{eq::BIVP_Approx}, respectively. Then there exist constants $M \geq 1$ and $\omega \in \R$ such that
    \begin{equation*}
        \begin{multlined}
             \|\state(t)-\stateapprox(t)\|_{\statespace}\le \varepsilon(t) \coloneqq \|\boprinv\|_{\linop{\controlspace}{\statespace}}\|\residualbc(t)\|_{\controlspace} + M \mexp{\omega t} \|\residualinit - \boprinv \residualbc(0)\|_{\statespace} \\ 
             + \int_{0}^t M \mexp{\omega (t-s)} \left(\|\A\boprinv \|_{\linop{\controlspace}{\statespace}} \|\residualbc(s) \|_{\controlspace} + \|\boprinv\|_{\linop{\controlspace}{\statespace}} \|\residualbcdot(s)\|_{\controlspace} + \|\residualevo(s)\|_{\statespace} \right) \ds.
        \end{multlined}
    \end{equation*}
\end{theorem}

\begin{proof}
	The assumptions together with \Cref{prop:wellPosed} guarantee unique mild solutions of~\eqref{eq::BIVP} and~\eqref{eq::BIVP_Approx}. We define the error $\stateerr \vcentcolon= \stateapprox - \state$ and notice that by construction, $\stateerr$ is the unique mild solution of
	\begin{equation*}
        \left\{\quad
        \begin{aligned}
            \dot{\stateerr} & = \A \stateerr + \residualevo \quad & \mathrm{in} \;\timeInt \times \Omega, \\
            \bop \stateerr &= \residualbc \quad & \mathrm{in}\; \timeInt \times \partial \Omega, \\ 
            \stateerr &= \residualinit \quad & \mathrm{in} \;\left\lbrace t=0 \right\rbrace \times \Omega.
        \end{aligned}\right.
    \end{equation*}
    Consequently, we have
    \begin{align}
    	\label{eqn:error:mildSol}
    	\stateerr(t) = \boprinv\residualbc(t) + \S(t)\big(\residualinit - \boprinv \residualbc(0)\big) + \int_0^t \S(t-s)\big( \residualevo(s) - \boprinv \residualbcdot(s) + \A\boprinv\residualbc(s)\big)\ds.
    \end{align}
    Let $M\geq 1$ and $\omega\in\R$ be the growth parameters of the semigroup  $\S(t)$, i.e., 
    \begin{align*}
   		\|\S(t)\|_{\linopo{\statespace}} \leq M\mathrm{e}^{\omega t}.
   	\end{align*}
   	Taking the norm of~\eqref{eqn:error:mildSol}, applying the triangle inequality, and using the monotonicity of the integral yields
    \begin{align*}
    	\|\stateerr(t)\|_{\statespace}  
        &\leq \|\boprinv\|_{\linop{\controlspace}{\statespace}}\|\residualbc(t)\|_{\controlspace} + M \mexp{\omega t} \|\residualinit - \boprinv \residualbc(0)\|_{\statespace}\\
        &\quad + \int_{0}^t M \mexp{\alpha (t-s)} \left(\|\A\boprinv \|_{\linop{\controlspace}{\statespace}} \|\residualbc(s) \|_{\controlspace} + \|\boprinv\|_{\linop{\controlspace}{\statespace}} \|\residualbcdot(s)\|_{\controlspace} + \|\residualevo(s)\|_{\statespace} \right) \d{s}. \qedhere
    \end{align*}
\end{proof}

In comparison to the error bound from \Cref{thm:errorBound} presented in \cite{HilU25}, wherein \ISS is employed to handle approximation errors on the boundary, the exponential stability (required by \ISS) of the underlying dynamical system is not necessary in the error estimator \Cref{cor::error_estimator}. In return, stronger assumptions are made on the regularity of the boundary perturbation. 
\begin{remark}
    \label{remark::temporal_derivative_boundary_error}
    Since the term $\|\residualbcdot(s)\|_{\controlspace}$ appears in the error bound, it seems reasonable from a \PINN context to include this term in the loss function for the training of the \PINN. 
\end{remark}

\begin{remark}
	If the operator $\boprinv$ is explicitly available, then the error bound can be improved by explicitly computing $\|\residualevo(s) - \boprinv \residualbcdot(s) + \A\boprinv \residualbc(s)\|_{\statespace}$ under the integral without further estimating as in the proof of \Cref{cor::error_estimator}.
\end{remark}

To use the error estimator from \Cref{thm:errorBound} practically, the norms $\|\boprinv\|_{\linop{\controlspace}{\statespace}}$, and $\|\A \boprinv\|_{\linop{\controlspace}{\statespace}}$ as well as the semigroup growth bound parameters $M$ and $\omega$ need to be computed. Since this is in generally not possible analytically, we show in the following how these values can be determined from numerical approximations. Let an approximation sequence of semigroups $(\approxstatespace, \Pn, \En)_{n \in \N}$ and for each $n$ accordingly discretized operators $\An \colon \approxstatespace \to \approxstatespace$ and $\bopn \colon \approxstatespace \to \controlspace$ be given. Further, assume that $\bopn$ has a right inverse denoted by $\bopnrinv \colon \controlspace \to D(\An) \subseteq \approxstatespace$, and that the operators converge in the strong operator topology, i.e.,
\begin{equation*}
    \begin{aligned}
        \En \An \Pn & \narrowinftySOT \A, \qquad\text{and}\qquad
        \En \bopnrinv & \narrowinftySOT \boprinv.
    \end{aligned}
\end{equation*}

Based on this setting, we propose to approximate $M$ and $\omega$ via the Trotter-Kato theorem, see \cite{ItoK98}, as follows.

\begin{theorem} 
    \label{prop::approx_growth_bound}
    Let $(\approxstatespace, \Pn, \En)_{n \in \N}$ be an approximation sequence of the Banach space~$\statespace$. Further, consider semigroups of operators $\S(t)$ on $\statespace$ and $\Sn(t)$ on $\approxstatespace$ for $n\in\N$ with respective growth bounds $\|\Sn(t)\|_{\linopo{\approxstatespace}} \le M_n \mexp{\omega_n t}$. Assume further that
    \begin{itemize}
        \item[(i)] $\Sn(t) \narrowinftySOT \S(t)$ and
        \item[(ii)] $(M_n)_{n\in \N}$ and $(\omega_{n})_{n\in \N}$ are Cauchy sequences.
    \end{itemize}
    Then $\|\S(t)\|_{\linopo{\statespace}} \le M^\star \mexp{\omega^\star t}$ with
    \begin{equation}
        \begin{aligned}
            \omega^\star &\coloneqq \lim_{n\to\infty} \omega_{n} , &
            M^\star &\coloneqq  \mup \mue \tilde{M} , &
            \tilde{M } & \coloneqq  \lim_{n \to \infty} M_n
        \end{aligned}
    \end{equation}
\end{theorem}

\begin{proof}
    For fixed $\state \in \statespace$ define $\staten \coloneqq \Pn \state$. Then, $\lim_{n\to\infty} \staten = \state$ and thus
	\begin{align*}
		\|\S(t) \state\|_{\statespace} &= \lim_{n\to\infty} \|\En \Sn(t) \staten\|_{\statespace} \leq \lim_{n\to\infty} \|\En\|_{\linop{\approxstatespace}{\statespace}} \|\Sn(t)\|_{\linopo{\approxstatespace}} \|\Pn\|_{\linop{\statespace}{\approxstatespace}} \|\state\|_{\statespace}\\
		&\leq \mup \mue \|\state\|_{\statespace} \lim_{n \to\infty} \|\Sn(t)\|_{\linopo{\approxstatespace}}.
	\end{align*}
	We immediately conclude $\|\S(t)\|_{\linopo{\statespace}} \leq \mup\mue \lim_{n\to\infty} \|\Sn(t)\|_{\linopo{\approxstatespace}} \leq M^\star \mexp{\omega^\star t}$.
\end{proof}

Let us emphasize that the assumption $\Sn(t)\narrowinftySOT \S(t)$ is critical, and for its validation in practice, we will use the Trotter-Kato theorem or variants thereof; see \cite{ItoK98}.

Now, the norms $\|\boprinv\|_{\linop{\controlspace}{\statespace}}$, and $\|\A \boprinv\|_{\linop{\controlspace}{\statespace}}$ remain. For this, we observe that since the operators $\A \boprinv$ and $\boprinv$ are bounded from $\controlspace$ to $\statespace$,
their norms can be bounded from above (according to Banach-Steinhaus \cite[Thm.~2.2]{Bre11}) using finite-dimensional approximations, i.e.,
\begin{align*}
    \|\A\boprinv\|_{\linop{\controlspace}{\statespace}} &\le \lim_{n \rightarrow \infty }\mue \|\An \bopnrinv\|_{\linop{\controlspace}{\approxstatespace}} , &
    \|\boprinv\|_{\linop{\controlspace}{\statespace}} &\le \lim_{n \rightarrow \infty } \mue \| \bopnrinv\|_{\linop{\controlspace}{\approxstatespace}}.
\end{align*}

Summarizing the previous discussion yields the following strategy to certify a \PINN approximation. Given a well-posed (in the sense of \Cref{prop:wellPosed}) dynamical system~\eqref{eq::BIVP}, we have to perform the following steps as preparation for the certification.
\begin{enumerate}
	\item \emph{Discretization}: Compute a discretization, i.e., fix the space $\approxstatespace$ and the corresponding approximations $\An$ and $\bopn$ of $\A$ and $\bop$. Construct an approximation sequence $(\approxstatespace, \Pn, \En)_{n \in \N}$ of the Banach space $\statespace$, and show that the sequence of semigroups generated by $\Angen = \An |_{\ker\bopn}$ converges to the semigroup generated by $\Agen = \A |_{\ker\bop}$.
	\item \emph{Estimation of the growth bounds}: Retrieve $M_n$ and $\omega_n$ from the Jordan normal form of $\An$. If $\An$ is diagonalizable, then we can simply set $M_n=1$ and $\omega_n = \real \left( \lambda_{\max}\right)$. Otherwise, i.e., if the matrix $\An$ is defective, we apply \cite[Thm.~10.12]{Hig08} and consider the Schur decomposition of $\An = Q (D + N) Q^\T$ to compute 
    \begin{equation*} 
        \begin{aligned}
        \|\Sn(t)\|_{\approxstatespace} &\le \mexp{\Real\left(\lambda_{\max}\right)t} \sum_{k=0}^\rho \frac{\|N\|^k t^k}{k!} \le \left(\sum_{k=0}^\rho \frac{\|N\|^k}{k!}\right) \mexp{\Real\left(\lambda_{\max}\right)t} \sum_{k=0}^\rho t^k \\&\le M_{n, 0} C_{n, \epsilon}\mexp{\Real\left(\lambda_{\max}\right) + \epsilon}.
        \end{aligned}
    \end{equation*}
    The polynomial in $t$ was absorbed in a slightly bigger growth bound and the additional factor $C_{n, \epsilon}$. Hence, $M_n = \sum_{k=0}^\rho \tfrac{\|N\|^k}{k!} C_{n, \epsilon}$ and $\omega_n = \real{(\lambda_{\max})}+\epsilon$. Alternative, eventually computationally more efficient and stable approaches can be found in \cite{MolV03}.
    Having computed $\omega_n$ and $M_n$ for fixed $n \in \N$, it remains to show that $(M_n)_{n\in\N}$ and $(\omega_n)_{n\in\N}$ are convergent.
	\item \emph{Estimation of operator norms}: Show that $\boprinv$ is a bounded operator from $\controlspace$ to $(D(\A),\|\cdot \|_\A)$ and compute $\|\A\boprinv\|_{\linop{\controlspace}{\statespace}} \le \lim_{n\to\infty} \mue\|\An \bopnrinv\|_{\linop{\controlspace}{\approxstatespace}}$ and $\|\boprinv\|_{\linop{\controlspace}{\statespace}} \le \lim_{n\to\infty} \mue\|\bopnrinv\|_{\linop{\controlspace}{\approxstatespace}}$.
\end{enumerate}
Then, a \PINN prediction at a time point $t$ can be certified by measuring~$\residualevo, \residualbc, \residualbcdot$, and $\residualinit$ and computing the upper bound of the prediction error from \Cref{cor::error_estimator}.
In case mathematical rigor is not required, then an error indicator can be computed by discretizing the problem, computing $M_n$, $\omega_n$, $\|\An\bopnrinv\|_{\linop{\controlspace}{\approxstatespace}}$, and $\|\bopnrinv\|_{\linop{\controlspace}{\approxstatespace}}$, and guess the corresponding limits from the discretized problem.


\section{Numerical examples}
\label{sec::NumEx}

In this section, we illustrate the results from the previous section using two examples. In the first example, presented in \Cref{subsec:heatEquation}, we discuss the one-dimensional heat equation, compute the approximate growth bounds from \Cref{prop::approx_growth_bound} and compare them with the exact results taken from the literature. While this certainly qualifies as an academic toy example, we discuss a more challenging example in \Cref{subsec:stokes}, where we consider the Stokes flow around a cylinder. Both problems are implemented as \PINNs and prepared the second example using the \emph{finite element} (\FE) toolbox \abbr{DOLFINx} v0.9.0 \cite{BarDDHHRRSSW23}. The full code can be found in 
\begin{center}
    \url{https://doi.org/10.5281/zenodo.16793705}.
\end{center}

\subsection{One-dimensional heat equation}
\label{subsec:heatEquation}

We consider the one-dimensional heat equation for $\Omega = [0,1]$ and $\timeInt = [0,0.5]$ given by
\begin{equation}
	\label{eqn:heatEquation}
	\left\{\qquad
    \begin{aligned} 
        \partial_t \state(t, x) & = \alpha \partial^2_x \state(t, x) \quad & \mathrm{for}\; (t, x) &\in \timeInt \times \Omega, \\ 
        \state(0, x) & = \stateinit(x) \quad & \mathrm{for}\; x &\in \Omega , \\ 
        \begin{bmatrix}
        	\state(t, 0)\\
        	\state(t, 1)
        \end{bmatrix} & = \statebc(t) \quad & \mathrm{for}\; t &\in \timeInt,\\
    \end{aligned}\right.
\end{equation}
for some $\alpha>0$ and $\statebc\colon\timeInt\to\R^2$. To cast~\eqref{eqn:heatEquation} in the form~\eqref{eq::BIVP}, we use the Hilbert spaces $\statespace = L^2(\Omega)$ and $\controlspace = \R^2$, and set $\A \state \vcentcolon= \alpha \partial_x^2 \state$ with $D(\A) = H^2(\Omega)$.
The operator $\bop\colon H^1(\Omega)\to \R^2$ is the Dirichlet trace operator, and hence it is easy to observe that $\ker(\bop) = H_0^1(\Omega) $.
Regarding \Cref{ass::well_posed}, we immediately obtain that $\A |_{\ker\bop}$ is the generator of a contraction semigroup with $D(\Agen) = H_0^1(\Omega) \cap H^2(\Omega)$ \cite[Sec. 7.4, Thm. 5]{Eva10}. Moreover, we can define
\begin{align*}
	\boprinv\colon \R^2\to D(\Agen),\qquad \begin{bmatrix}
		\state_{b,1}\\
		\state_{b,2}
	\end{bmatrix}\mapsto \bigg(\Omega\to \R, \quad x\mapsto x \state_{b,2} + (1-x)\state_{b,1}\bigg),
\end{align*}
such that we can conclude that \Cref{ass::well_posed} is satisfied. Further, we obtain $\A\boprinv = 0$, which simplifies the error bound from \Cref{cor::error_estimator}. Moreover, straightforward calculations show $\|\boprinv\|_{\linop{\controlspace}{\statespace}} \leq \tfrac{2}{3}$.
%

With these preparations, we now perform the next steps for the error certification outlined in the end of \Cref{sec::theory}. Following similar ideas as in \cite[Sec.~4.1]{ItoK98}, we use a finite difference scheme for the numerical approximation. In more detail, for $n\in\N$, we define $\deltacoordn \vcentcolon= \tfrac{1}{n+1}$ and $\coordnk{k} \vcentcolon= k \deltacoordn$ such that $\coordnk{0} = 0$ and $\coordnk{n+1} = 1$. We then set $\approxstatespace = \R^{n}$ with the norm
\begin{align*}
	\|\staten\|_{\approxstatespace}^2 \vcentcolon= \deltacoordn \sum_{k=1}^{n} | (\staten)_k |^2
\end{align*}
and accordingly defined inner product, and
\begin{align*}
	(\Pn \state)_k &=
		\frac{1}{\deltacoordn} \int_{\coordnk{k}-\deltacoordn/2}^{\coordnk{k}+\deltacoordn/2} \state(s) \ds, \; \text{for $1 < k < n$},\\
		E_n \staten & =\sum_{k=1}^n \statenk \chi_{[\coordnk{k}-\deltacoordn/2, \,  \coordnk{k}+\deltacoordn/2)},
\end{align*}
where $\chi_{[a,b)}$ is the indicator function for the interval $[a,b)$.
Using centred finite differences for the second derivative, we define the matrices
\begin{align*}
	\Angen &= \frac{\alpha}{\deltacoordn^2} \begin{bmatrix}
		-2 & 1 & 0 & \dots & 0 \\
		1 & -2 & 1 & & \vdots \\ 
		0 & 1 & -2 & & \vdots \\ 
		\vdots &  & \ddots & \ddots& \vdots \\
		& \cdots & 0 & -1 &2
	\end{bmatrix}\in\R^{n\times n}, &
	D_n &= \frac{\alpha}{\deltacoordn^2} \begin{bmatrix}
		1 & 0\\
		0 & 0\\
		0 & \vdots \\
		\vdots & \vdots\\
		\vdots & 0 \\
		0 & 1
	\end{bmatrix}\in \R^{n\times 2}
\end{align*}
and consider the discretized problem
\begin{equation*}
    \left\{\quad\begin{aligned}
        \statendot(t) &= \Angen \staten(t) + D_n \statebc(t),\\ 
        \staten(0) & = \Pn \stateinit. \\ 
    \end{aligned}\right.
\end{equation*}
It is well known that the eigenvalues of $\Angen$ are given by $\lambda_j = -\tfrac{4\alpha}{\deltacoordn^2} \sin^2\big(\tfrac{\pi j}{2(n+1)}\big)$ for $j=1,\ldots,n$ \cite[Sec. 2.10]{LeV07}. Since $\Angen$ is symmetric, we can thus set $M_n = 1$ and
\begin{align*}
	\omega_n = \max_{j=1,\ldots,n} \lambda_j = -\alpha \left(2(n+1) \sin\left(\tfrac{\pi}{2(n+1)}\right)\right)^2
\end{align*}

\begin{lemma}
	\label{lem:heatEquation}
	Let the setting be as in this subsection and let $\Sn(t) = \mexp{\Angen t}$ be the semigroup associated with the matrix $\Angen$ for $n\in\N$. Then
	\begin{enumerate}
		\item the sequence $(\approxstatespace, \Pn, \En)_{n\in \N}$ is an approximation sequence according to \Cref{def::approx_sequence} with  $\mup = \mue = 1$,
		\item $\Sn(t) \narrowinftySOT \S(t)$, and 
		\item $\lim_{n\to\infty} M_n = 1$, $\omega^\star = \lim_{n\to\infty} \omega_n = -\alpha \pi^2$, and $\omega^\star \leq \omega_n$ for all $n\in\N$.
	\end{enumerate}
\end{lemma}

\begin{proof}
	We proof each item separately.
	\begin{enumerate}
		\item Direct computation reveals $\Pn\En = \calI_{\approxstatespace}$ for $n\in\N$ and $\langle \En u, \En v\rangle_{\statespace} = \langle u,v\rangle_{\approxstatespace}$, i.e.~$\|\En\|_{\linop{\approxstatespace}{\statespace}} = 1$. 
			Further, using $\big(\int_a^b |f(s)| \ds)^2 \leq (b-a) \int_a^b f^2(s) \ds$, we obtain $\|\Pn \state\|_{\approxstatespace} \leq \|\state\|_\statespace$ and thus $\|\Pn\|_{\linop{\statespace}{\approxstatespace}} = 1$.
			Moreover, we notice that the first property in \Cref{def::approx_sequence} can be shown using the Lebesgue differentiation theorem. 
			Firstly, we observe that $\|\En\staten\|_{\statespace}^2 = \deltacoordn \|\staten\|_{\ell_2}^2 = \|\staten\|_{\approxstatespace}^2$ and
			\begin{align*}
				 \En\Pn\state (s) \narrowinfty \state(s)
			\end{align*}
			almost everywhere by the Lebesgue differentiation theorem. Combining this yields 
			\begin{equation*}
				\lim_{n\rightarrow \infty} \|\Pn \state \|_{\approxstatespace}^2 = \lim_{n \rightarrow\infty} \|\En\Pn\state \|_{\statespace}^2 = \|\state \|_{\statespace}^2 .
			\end{equation*}
		\item To show $\Sn(t) \narrowinftySOT \S(t)$, we will use the Trotter-Kato theoren \cite[Thm.~2.1]{ItoK98}. To simplify notation, we set $\Agen \vcentcolon= \A |_{\ker \bop}$. Instead of working directly with the resolvent, we will use \cite[Prop.~3.1]{ItoK98}. In more detail, we need to show two properties, namely
		\begin{enumerate}
			\item[(a)] that there exists a subset $D \subseteq D(\Agen)$ such that the closure of $D$ is equal to $\statespace$ and such that $\overline{(\lambda_0 \mathcal{I} - \Agen)D} = \statespace$ for some $\lambda_0 > \omega$, wherein $\omega$ denotes the growth factor of the semigroup generated by $\Agen$, and
			\item[(b)] that for all $\state \in D$ there exists a sequence $(\stategeneralapprox)_{n\in\N}$ with $\stategeneralapprox \in D(\Angen)$ such that 
			\begin{equation*}
    			\lim_{n\to\infty} \|\En \stategeneralapprox - \stategeneral \|_{\statespace} = 0 \qquad \text{and}\qquad 
    			\lim_{n\to\infty} \| \En \Angen \stategeneralapprox - \Agen \stategeneral \|_{\statespace} = 0,
			\end{equation*}
		\end{enumerate} 
		For (a), consider $D \vcentcolon= C_c^2(\Omega)\subseteq D(\Agen)$, then since $C_c([0, 1])$ and $C^2([0, 1])$ are dense in $L^2([0, 1])$ (\!\!\cite[Thm.~4.12]{Bre11} and the Weierstrass approximation theorem) the conditions are fulfilled.
		For (b), let $\state \in D$ and set $\stategeneralapprox \vcentcolon= (\stategeneral(x_{n,1}),\ldots, \stategeneral(x_{n,n}))$. Then, $\|\En \stategeneralapprox - \stategeneral \|_{L^2}^2 \narrowinfty 0$ as shown in \cite[Sec. 4.1, Case 2]{ItoK98}, furthermore, we compute
		\begingroup
		\allowdisplaybreaks
		\begin{align*}
		    \| \En \Angen \stategeneralapprox - \Agen \stategeneral \|_{L^2}^2 
			& = \int_0^1 \left( (\En \Angen \stategeneralapprox - \Agen \stategeneral) (x)\right)^2  \d{x} \\
		    & = \sum_{k=1}^{n} \int_{x_{k-1}}^{x_k} \left( \frac{\stategeneral(x_{k+1})-2\stategeneral(x_k) +\stategeneral(x_{k-1})}{(\deltacoordn)^2}  - \stategeneral''(x) \right)^2 \d{x} \\
		    & = \sum_{k=1}^{n} \int_{x_{k-1}}^{x_k} \left( \frac{1}{(\deltacoordn)^2} \int_{x_k}^{x_{k+1}} \int_{\tau- \deltacoordn}^{\tau} \stategeneral''(\sigma) \d{\sigma}\d{\tau} - \stategeneral''(x) \right)^2 \d{x} \\
		    & \le \frac{1}{(\deltacoordn)^4} \sum_{k=1}^{n} \int_{x_{k-1}}^{x_k} \left( \int_{x_k}^{x_{k+1}} \int_{\tau- \deltacoordn}^{\tau} | \stategeneral''(\sigma) - \stategeneral''(x) |\d{\sigma}\d{\tau}\right)^2 \d{x} \\
		    & \overset{(1)}{\le} \frac{1}{(\deltacoordn)^4} \sum_{k=1}^{n} \int_{x_{k-1}}^{x_k} \left( (\deltacoordn)^2 \int_{x_k}^{x_{k+1}} \int_{\tau- \deltacoordn}^{\tau} | \stategeneral''(\sigma) - \stategeneral''(x) |^2 \d{\sigma}\d{\tau}\right) \d{x} \\
		    & \overset{(2)}{\le} \frac{1}{(\deltacoordn)^2}  \sum_{k=1}^{n} \int_{x_{k-1}}^{x_k} \left( \int_{x_k}^{x_{k+1}} \int_{\tau- \deltacoordn}^{\tau} \omega(\stategeneral'', \deltacoordn)^2 \d{\sigma}\d{\tau}\right) \d{x} \\
		    & \le \frac{1}{(\deltacoordn)^2} \sum_{k=1}^{n}  \int_{x_{k-1}}^{x_k} \left( (\deltacoordn)^2 \omega(\stategeneral'', \deltacoordn)^2 \right) \d{x} \\
		    & \le \omega(\stategeneral'', \deltacoordn)^2 
		\end{align*}
		\endgroup
		wherein $(1)$ holds because of the $\Lone$-$\Ltwo$ integral inequality, $(2)$ uses the continuity modulus $\omega$ of $\stategeneral''$. Since the continuity modulus converges to 0 for $\deltacoordn \rightarrow 0$, this shows the remaining inequality.
		We conclude that the assumptions of \cite[Prop.~3.1]{ItoK98} are satisfied, which in turn allows us to use the Trotter-Kato theorem \cite[Thm.~2.1]{ItoK98} to establish convergence of the sequence of semigroups.						

		\item This follows immediately from
			\begin{align*}
				\lim_{n\to\infty} \omega_n = \lim_{x\to 0} -\alpha \left(\pi \frac{\sin(x)}{x}\right)^2 = -\alpha \pi^2.\\[-3em]
			\end{align*}
			\qedhere
	\end{enumerate}
\end{proof}

Let us emphasize that the limit $\omega^\star = -\alpha\pi^2$ resembles the growth bound of the semigroup~$\S(t)$; see for instance \cite[VI.8.9]{EngN00}. The convergence behavior of $\omega_n$ is plotted in \Cref{fig::heat_equation_growth_bound}. In particular, we observe that for this particular approximation, $\omega_n \geq \omega^\star$ for all~$n\in\N$. 

\begin{figure}[ht]
    \centering
    \begin{tikzpicture}
        \begin{axis}[
            xlabel={discretization steps n},
            ylabel={$\omega_n$},
            xmin=0, xmax=8000,
            ymin=-9.88, ymax=-9.82,
            legend pos=north east,
            ymajorgrids=true,
            grid=both,
			grid style={line width=.1pt, dashed, draw=gray!10},
			major grid style={line width=.2pt,draw=gray!50},
			axis lines*=left,
			axis line style={line width=\plotlinewidth},
            width=0.8 \linewidth,
            height=0.3\linewidth
        ]
        
        \addplot[line width=\plotlinewidth,solid,color=c_teal] %
        table[x=n,y=omega,col sep=comma]{figures/heat_equation/omega_growth_bound_limit.csv};
        \addlegendentry{$\omega_n$}
        \addplot[line width=\plotlinewidth, dashed, color=c_wine]
            coordinates {
            (0, -9.8696) (8000, -9.8696)
            };
            \addlegendentry{$\omega^\star$}
        \end{axis}
    \end{tikzpicture}
    \caption{The growth bound $\omega_n$ for the discretized systems with $n$ subdivisions for the 1D heat equation in comparison to the growth bound found in literature $\omega^\star$ \cite[VI.8.9]{EngN00}.}
    \label{fig::heat_equation_growth_bound}
\end{figure}
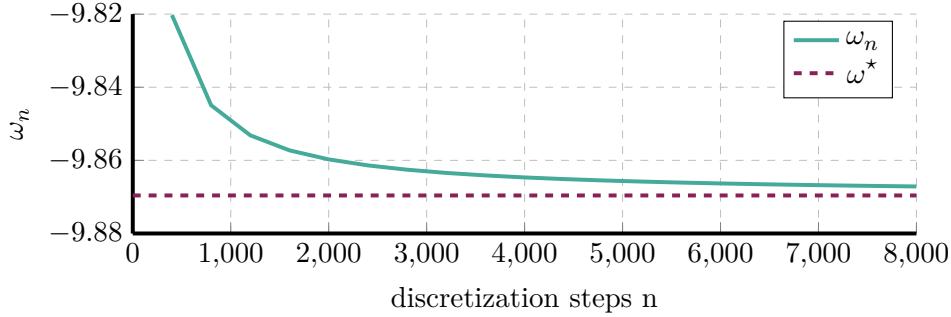

In summary, we obtain the following error bound for the one-dimensional heat equation.
\begin{corollary}
	\label{cor:heatEquationErrBound}
    Consider the one-dimensional heat equation~\eqref{eqn:heatEquation} and let $\stateapprox$ be a numerical approximation of the mild solution $\state$. For a numerical solution to the heat equation on $\Omega = [0, 1]$, the prediction error of the numerical approximation $\tilde{\state}(t)$ to the true solution~$\state(t)$ can be bounded by 
    \begin{equation}
    	\label{eqn:errBound:HeatEquation}
    	\begin{aligned}
    		\| \state(t) - \tilde{\state}(t) \|_{\statespace} 
		& \le \tfrac{2}{3} \|\delta_b(t)\|_{\controlspace} + \mexp{-\alpha \pi^2 t}\|\delta_0 - \boprinv \delta_b(0)\|_{\statespace} \\ 
		& \phantom{=} + \int_0^t \mexp{-\alpha \pi^2 (t-s)}\left(\tfrac{2}{3}\|\dot{\delta}_b(s)\|_{\controlspace} + \|\delta(s)\|_{\statespace} \right)\d{s}.
    	\end{aligned}
    \end{equation}
\end{corollary}

We now demonstrate the validity of the proposed error estimator by applying it to a \PINN trained to solve the one-dimensional heat equation for $\alpha = 0.2$. For this purpose, we extend the implementation published with \cite{HilU25} by incorporating the new error estimator, train a neural network, and evaluate it accordingly. The neural network has 4 hidden layers with 10 neurons each. The activation function used is the hyperbolic tangent. The network was trained for $\num{10000}$ epochs with the L-BFGS optimizer from \cite{LiuN89}. The loss is given by 
\begin{align*}
    L(\stateapprox) &= \frac{\alpha_{\mathrm{evo}}}{N_{\mathrm{evo}}} \sum_{i = 0}^{N_{\mathrm{evo}}} \left|\partial_t \stateapprox(t^{\mathrm{evo}}_i, x^{\mathrm{evo}}_i) - 0.2\cdot \partial_x^2 \stateapprox(t^{\mathrm{evo}}_i, x^{\mathrm{evo}}_i) \right| + \frac{\alpha_{\mathrm{init}}}{N_{\mathrm{init}}} \sum_{i = 0}^{N_{\mathrm{init}}} \left|\stateapprox(0, x^{\mathrm{init}}_i) - \stateinit(x_i) \right| \\
    & \phantom{=} + \frac{\alpha_{\mathrm{bc, 1}}}{N_{\mathrm{bc}}} \sum_{i = 0}^{N_{\mathrm{bc}}} \left| \stateapprox(t^{\mathrm{bc}}_i, x^{\mathrm{bc}}_i) \right| + \frac{\alpha_{\mathrm{bc, 2}}}{N_{\mathrm{bc}}} \sum_{i = 0}^{N_{\mathrm{bc}}} \left| \partial_t \stateapprox(t^{\mathrm{bc}}_i, x^{\mathrm{bc}}_i) \right|.
\end{align*}
The last two loss terms correspond to the boundary error and its derivative, in particular $x^{\mathrm{bc}}_i \in \{ 0, 1\}$. Note that in agreement with the findings in \Cref{sec::theory} we explicitly add the temporal derivative of the boundary condition to the loss term. The collocation and training points used were generated using latin hypercube sampling \cite{McKBC79} with $N_{\mathrm{evo}} = \num{5000}$, $N_{\mathrm{init}} = \num{200}$ and $N_{\mathrm{bc}} = \num{100}$.
The individual loss terms are weighted by $\alpha_{\mathrm{init}} = \alpha_{\mathrm{evo}} = \tfrac{1}{2}$, $\alpha_{\mathrm{bc, 1}} = \num{20}$, and $\alpha_{\mathrm{bc, 2}} = \num{200}$. The resulting \PINN prediction shows good agreement with the true solution across the entire space-time domain. 

To evaluate the prediction error of the \PINN, we denote the different contributions to the error estimator~\eqref{eqn:errBound:HeatEquation} as
\begin{equation}\label{eq::heat_contrib}
	\begin{aligned}
		\varepsilon_{\mathrm{init}} &\coloneqq \mexp{-\alpha \pi^2 t}\|\delta_0 - \boprinv \delta_b(0)\|_{\statespace}, \quad &
		\varepsilon_{b, [t]} &\coloneqq \tfrac{2}{3} \| \delta_b(t)\|_{\controlspace},\\
		\varepsilon_{b, \int} &\coloneqq  \int_0^t \mexp{-\alpha \pi^2 (t-s)}\tfrac{2}{3}\|\dot{\delta}_b(s)\|_{\controlspace} \d{s}, \quad &
		\varepsilon_{\mathrm{evo}} &\coloneqq  \int_0^t \mexp{-\alpha \pi^2 (t-s)} \|\delta(s)\|_{\statespace}\d{s}, \\
	\end{aligned}
\end{equation}
such that the overall error estimator is given by
\begin{align*}
	\varepsilon_{\text{tot}} =  \varepsilon_{b, [t]} + \varepsilon_{\mathrm{init}} + \varepsilon_{b, \int} + \varepsilon_{\mathrm{evo}};
\end{align*}
cf.~\Cref{cor:heatEquationErrBound}.
The results are displayed in \Cref{fig::heat_error2}, where it is evident that the error estimator (solid purple line) provides a true upper bound to the actual error of the \PINN prediction (solid green line). The figure also highlights that omitting the boundary error contributions leads to a substantial underestimation of the prediction error. Vice versa, the consideration of the boundary error in the prediction, as we propose it in this contribution, is essential for the validity of the error estimator.
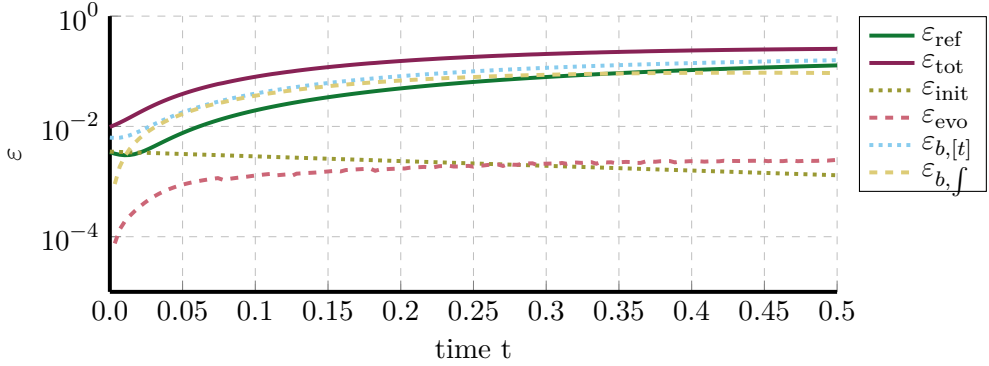
\begin{figure}[ht]
    \centering
    \begin{tikzpicture}
        \begin{axis}[
            xlabel={time t},
            ylabel={$\varepsilon$},
            xmin=0, xmax=0.5,
            legend pos=outer north east,
            legend style = {cells={anchor=west}},
            ymajorgrids=true,
            grid=both,
			grid style={line width=.1pt, dashed, draw=gray!10},
			major grid style={line width=.2pt,draw=gray!50},
			axis lines*=left,
			axis line style={line width=\plotlinewidth},
			ymode=log,
			ymin = 1e-5, ymax = 1,
            grid style=dashed,
            width=0.75\linewidth,
            height=0.35\linewidth,
            xtick={0.0, 0.05, 0.1, 0.15, 0.2, 0.25, 0.3, 0.35, 0.4, 0.45, 0.5},
            xticklabels={0.0, 0.05, 0.1, 0.15, 0.2, 0.25, 0.3, 0.35, 0.4, 0.45, 0.5},
        ]
        
        \addplot[line width=\plotlinewidth,solid,color=c_ref] %
        table[x=t,y=E_ref,col sep=comma]{figures/heat_equation/heat_equation_soft_bc_run_tanh_test_data_t_error_over_time.csv};
        \addlegendentry{$\varepsilon_{\text{ref}}$}
        
        \addplot[line width=\plotlinewidth,solid,color=c_tot] %
        table[x=t,y=E_tot,col sep=comma]{figures/heat_equation/heat_equation_soft_bc_run_tanh_test_data_t_error_over_time.csv};
        \addlegendentry{$\varepsilon_{\text{tot}}$}

        \addplot[line width=\plotlinewidth,dotted,color=c_init] %
        table[x=t,y=E_init,col sep=comma]{figures/heat_equation/heat_equation_soft_bc_run_tanh_test_data_t_error_over_time.csv};
        \addlegendentry{$\varepsilon_{\text{init}}$}

        \addplot[line width=\plotlinewidth,dashed,color=c_evo] %
        table[x=t,y=E_PI,col sep=comma]{figures/heat_equation/heat_equation_soft_bc_run_tanh_test_data_t_error_over_time.csv};
        \addlegendentry{$\varepsilon_{\text{evo}}$}

        \addplot[line width=\plotlinewidth,dotted,color=c_bt] %
        table[x=t,y=E_bc_sum_ubt,col sep=comma]{figures/heat_equation/heat_equation_soft_bc_run_tanh_test_data_t_error_over_time.csv};
        \addlegendentry{$\varepsilon_{b, [t]}$}
        
        \addplot[line width=\plotlinewidth,dashed,color=c_int_1] %
        table[x=t,y=E_bc_int_ubdot,col sep=comma]{figures/heat_equation/heat_equation_soft_bc_run_tanh_test_data_t_error_over_time.csv};
        \addlegendentry{$\varepsilon_{b, \int}$}
        
        \end{axis}
    \end{tikzpicture}
    \caption{Contributions~\eqref{eq::heat_contrib} of the error estimator, their sum $\varepsilon_{\text{tot}}$ in comparison to the true error $\varepsilon_{\text{ref}}$.}
    \label{fig::heat_error2}
\end{figure}

The error estimator and the first results for the heat equation suggests that including the temporal derivative of the boundary errors additionally to the losses used conventionally for \PINNs with soft boundary constraints allows us to fine-tune the training parametrization. Beyond penalizing the boundary error significantly (in our case 40 times more than the error of the state evolution) as used previously in \cite{HilU25}, we investigate numerically how the relation 
\begin{equation*} 
    \lambda_{\mathrm{bc}} \coloneqq \alpha_{\mathrm{bc, 2}} / \alpha_{\mathrm{bc, 1}}
\end{equation*}
affects the training result and the error estimator. The factor $\alpha_{\mathrm{bc, 1}}$ is kept constant at $20$. We then investigate
\begin{equation*} 
    \lambda_{\mathrm{bc}} \in \left\lbrace \tfrac{1}{10}, 1, 3.14, 10 \right\rbrace.
\end{equation*}
The value $\lambda_{\mathrm{bc}} = 3.14$ was chosen to equilibrate the weightings of the two boundary error contributions in the error estimator, i.e.,
\begin{equation*}
    \lambda_{\mathrm{bc}} = 3.14 \approx \frac{\tfrac{2}{3}}{\int_0^t \mexp{-\alpha \pi^2 (t-s)}\tfrac{2}{3}\ds}
\end{equation*}
for the latest time of the time range, i.e., $t = 0.5$. The ratio between the two boundary error contributions over time for the prediction error are presented in \Cref{fig::error_ratio_lambda}. We do not observe a significant modification in the ratios between the two error contributions for varying $\lambda_{\mathrm{bc}}$. This is reasonable as we do not expect the approximation by the \PINN to be highly oscillating in time. In cases where high temporal oscillations in time are expected, a suitable weighting of the loss terms might be advantageous during training to ensure that the contributions to the error estimator are equally weighted.

\begin{figure}
    \centering 
    \begin{tikzpicture}
        \begin{axis}[
            xlabel={time t},
            ylabel={$\frac{\varepsilon_{\mathrm{b}, [t]}}{\varepsilon_{\mathrm{b}, \int}}$},
            xmin=0.025, xmax=0.5,
            legend pos=outer north east,
            legend style={
        		cells={anchor=west},
   			},
            ymajorgrids=true,
            grid=both,
			grid style={line width=.1pt, dashed, draw=gray!10},
			major grid style={line width=.2pt,draw=gray!50},
			axis lines*=left,
			axis line style={line width=\plotlinewidth},
            ylabel style={rotate=-90},
            xtick={0.0, 0.05, 0.1, 0.15, 0.2, 0.25, 0.3, 0.35, 0.4, 0.45, 0.5},
            xticklabels={0.0, 0.05, 0.1, 0.15, 0.2, 0.25, 0.3, 0.35, 0.4, 0.45, 0.5},
            grid style=dashed,
            width=0.75\linewidth,
            height=0.3\linewidth
        ]

        \addplot[line width=\plotlinewidth,solid,color=c_green] %
        table[x=t,y=bc_sububt_by_intubdot,col sep=comma]{figures/heat_equation/heat_equation_soft_bc_param5_run_tanh_test_data_t_error_over_time.csv};
        \addlegendentry{$\lambda_{\mathrm{bc}} = 10^2$}
        
        \addplot[line width=\plotlinewidth,dotted,color=c_rose] %
        table[x=t,y=bc_sububt_by_intubdot,col sep=comma]{figures/heat_equation/heat_equation_soft_bc_run_tanh_test_data_t_error_over_time.csv};
        \addlegendentry{$\lambda_{\mathrm{bc}} = 10^1$}

        \addplot[line width=\plotlinewidth,dashed,color=c_rose] %
        table[x=t,y=bc_sububt_by_intubdot,col sep=comma]{figures/heat_equation/heat_equation_soft_bc_param3_run_tanh_test_data_t_error_over_time.csv};
        \addlegendentry{$\lambda_{\mathrm{bc}} = 3.14$}

        \addplot[line width=\plotlinewidth,solid,color=c_wine] %
        table[x=t,y=bc_sububt_by_intubdot,col sep=comma]{figures/heat_equation/heat_equation_soft_bc_param2_run_tanh_test_data_t_error_over_time.csv};
        \addlegendentry{$\lambda_{\mathrm{bc}} = 10^0$}

        \addplot[line width=\plotlinewidth,dotted,color=c_teal] %
        table[x=t,y=bc_sububt_by_intubdot,col sep=comma]{figures/heat_equation/heat_equation_soft_bc_param4_run_tanh_test_data_t_error_over_time.csv};
        \addlegendentry{$\lambda_{\mathrm{bc}} = 10^{-1}$}
        \end{axis}
    \end{tikzpicture}
    \caption{Ratio between the boundary error contributions for five neural networks trained with different $\lambda_{\mathrm{bc}}$.}
    \label{fig::error_ratio_lambda}
\end{figure}
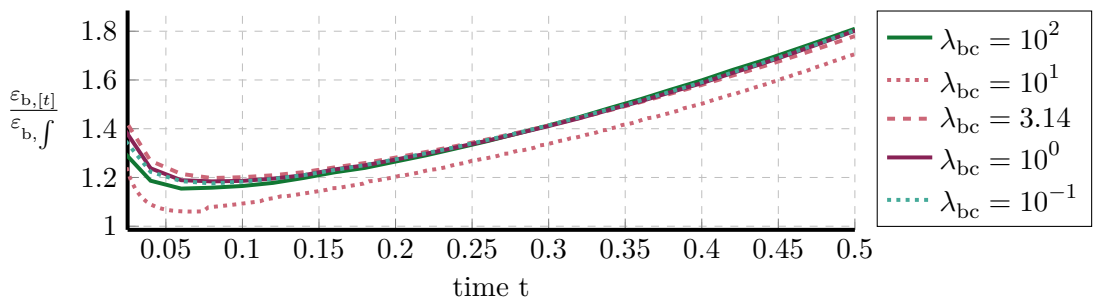

\subsection{Stokes flow around a cylinder}
\label{subsec:stokes}

In this subsection, we use the two-dimensional Stokes equation 
\begin{equation}
	\label{eqn:Stokes}
	\left\{\qquad
    \begin{aligned}
        \partial_t \vel + \nabla p - \mu \Delta \vel  &= 0, \quad \text{in }  \timeInt\times\Omega,\\
        \nabla \cdot \vel &= 0, \quad \text{in } \timeInt\times\Omega.
    \end{aligned}
    \right.
\end{equation}
with $\vel(t) \in \R^2$ and $p(t) \in \R$ denoting the velocity and pressure, respectively, on a domain with an obstacle to illustrate the applicability of the presented results for a problem on a complicated domain. The viscosity is set to $\mu = 0.001$. To fit the Stokes equation~\eqref{eqn:Stokes} into the framework described in the previous sections, we use the projection defined by the Helmholtz-Hodge decomposition \cite{ChoM93}. In more detail, let $P\colon\Ltwo(\Omega)\to \Ltwo_\sigma(\Omega)$ denote the projection onto the subspace of divergence free square integrable functions. Then, the Stokes operator is given as $\A = \mu P\Delta$ and the problem reads 
\begin{equation*}
    \partial_t\vel = \mu P\Delta\vel \quad \text{in }\Omega
\end{equation*}
which matches the desired form \eqref{eq::BIVP}. To simplify the estimation of $\omega, M, \|\boprinv\|_{\linop{\controlspace}{\statespace}}$, and $\|\A\boprinv\|_{\linop{\controlspace}{\statespace}}$, we perform all computations for their estimation in the space $\Ltwo(\Omega)$, i.e., we neglect the projection $P$. Since $P$ is a projection, the computed quantities will be upper bounds on the desired values. 

The domain under investigation, inspired by the benchmark \cite{SchTDKR96}, is depicted in \Cref{fig::stokesdomain}. The problem has Dirichlet boundary conditions on the walls, the inlet, and at the obstacle, the outlet has Neumann boundary conditions. The inlet is time-dependent and given by
\begin{equation*}
    \vel(x=0, y, t) = \left(\begin{matrix}
        6 \sin\left(\tfrac{\pi t}{8}\right)\tfrac{y (H-y)}{H^2}\\
        0
    \end{matrix}\right)
\end{equation*}
where $H$ denotes the height of the domain. The full Navier-Stokes equations for a similar geometry were implemented as a tutorial in the framework \abbr{DOLFINx} \cite{BarDDHHRRSSW23}, which is slightly modified to produce a suitable reference solution for the Stokes equation under consideration. 

\begin{figure}
	\centering
	\begin{subfigure}[t]{.43\linewidth}
		\centering
		\begin{tikzpicture}[scale=4]
            \def\L{1.2}     
            \def\H{0.41}    
            \def\r{0.05}    
            \def\cx{0.2}    
            \def\cy{0.2}    

            \draw[fill=gray!30, draw=white] (0,0) rectangle (\L, -0.1);
            \draw[fill=gray!30, draw=white] (0,\H) rectangle (\L, \H+0.1);
            \draw[thick] (0,0) rectangle (\L,\H);

            \draw[fill=gray!30, draw=black] (\cx,\cy) circle (\r);

            \draw[->] (0,0) -- (\L+0.1,0) node[right] {$x$};
            \draw[->] (0,0) -- (0,\H+0.1) node[above] {$y$};

            \draw[<->] (0,-0.05) -- (\L,-0.05) node[font=\tiny, midway, below] {$L = 1.2$};
            \draw[<->] (0.85,0) -- (0.85,\H) node[font=\tiny, midway, left] {$H = 0.41$};

            \draw[<->] (\cx, \cy) -- (\cx, \cy-0.05) node[font=\tiny, midway, below] {$r = 0.05$};
            \node at (\cx,\cy+0.1) [font=\tiny, anchor=center] {Obstacle};

            \node[font=\tiny, anchor=east] at (0,\H/2) {Inlet};
            \node[font=\tiny, anchor=west] at (\L,\H/2) {Outlet};
        \end{tikzpicture}
    	\caption{The domain $\Omega$ with homogeneous {Dirichlet} boundary conditions at the grey areas, time-dependent Dirichlet boundary condition at the inlet, and homogeneous Neumann boundary conditions at the outlet.}
	\end{subfigure}\hfill
	\begin{subfigure}[t]{.53\linewidth}
		\centering    
        \begin{tikzpicture}[scale=4]
            \def\L{1.2}     
            \def\dL{0.5}     
            \def\H{0.41}    
            \def\r{0.05}    
            \def\cx{0.2}    
            \def\cy{0.2}    

            \draw[fill=gray!30, draw=white] (0-\dL,0) rectangle (\L, -0.1);
            \draw[fill=gray!30, draw=white] (0-\dL,\H) rectangle (\L, \H+0.1);
            \draw[thick] (0,0) rectangle (\L,\H);
            \draw[thick, pattern=north east lines, pattern color=c_teal] (0,0) rectangle (-\dL,\H);

            \draw[fill=gray!30, draw=black] (\cx,\cy) circle (\r);

            \draw[->] (0,0) -- (\L+0.1,0) node[right] {$x$};
            \draw[->] (0,0) -- (0,\H+0.1) node[above] {$y$};

            \draw[<->] (0,-0.05) -- (\L,-0.05) node[font=\tiny, midway, below] {$L = 1.2$};
            \draw[<->] (-\dL,-0.05) -- (0,-0.05) node[font=\tiny, midway, below] {$\delta L = 0.5$};
            \draw[<->] (0.85,0) -- (0.85,\H) node[font=\tiny, midway, left] {$H = 0.41$};

            \draw[<->] (\cx, \cy) -- (\cx, \cy-0.05) node[font=\tiny, midway, below] {$r = 0.05$};
            \node at (\cx,\cy+0.1) [font=\tiny, anchor=center] {Obstacle};

            \node[font=\tiny, anchor=east] at (0,\H/2) {Inlet};
            \node[font=\tiny, anchor=west] at (\L,\H/2) {Outlet};
        \end{tikzpicture}
        \caption{The domain is extended by the region shaded in green to enable the computation of harmonic feature embeddings. Neumann boundary conditions are imposed on the modified inlet and the original outlet of the extended domain.}
    \end{subfigure}
    \caption{Computational setup for the two-dimensional Stokes flow around a cylinder.}
    \label{fig::stokesdomain}
\end{figure}
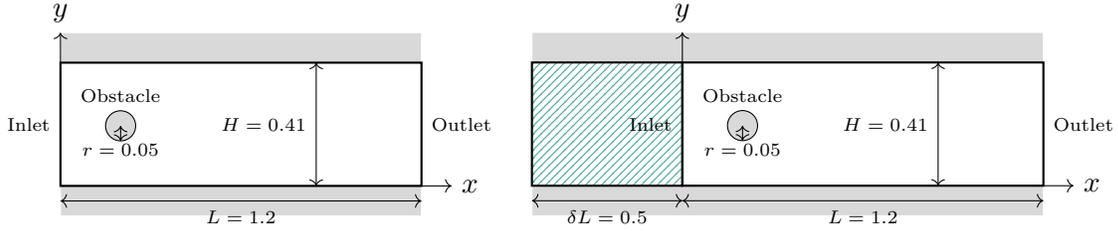

\subsubsection{Norm approximation of the semigroup and operators}
\label{subsub::ParamDeterminationStokes}

To derive the necessary values for the growth bounds and the operator norms, we leverage an \FE discretization, such that we can compute $\|\An\|_{\linop{\approxstatespace}{\approxstatespace}}$ and, with a small additional construction, $\|\An \bopnrinv \|_{\linop{\controlspace}{\approxstatespace}}$. For this, consider the weak formulation of the problem without boundary conditions, i.e.,
\begin{equation*}
    \left\langle \partial_{t} u , v \right\rangle = a_S(u, v),
\end{equation*}
where we write for simplicity $a_S$ for the bilinear form associated to the Stokes operator without considering boundary conditions and $\langle \cdot, \cdot \rangle$ for the duality pairing. The matrix representation of the bilinear form $a_S$ without enforcing boundary conditions yields $\An$. 
Conventionally, Dirichlet boundary conditions are enforced directly (c.f.~\cite{ZieT13}) and Neumann boundary conditions are enforced naturally. Nonetheless, since we explicitly need the spatially discretized boundary operator $\bopn$ and a right-inverse $\bopnrinv$, we proceed differently and construct these operators explicitly. The technical derivation of these matrices is presented in \Cref{app::boundarymat}. 

We use \abbr{DOLFINx} to generate meshes of varying resolution to compute~$\|\An \bopnrinv \|_{\linop{\controlspace}{\approxstatespace}}$ and $\|\bopnrinv \|_{\linop{\controlspace}{\approxstatespace}}$ for different numbers of vertices in the mesh $n$. The cell sizes are controlled during meshing by restricting the maximally allowed cell size. To observe the convergence of the numerically determined operator norms, we depict the results in \Cref{fig::stokes_params} and observe the following (estimated) limit values
\begin{equation*}
    \begin{aligned}
        \omega^\ast &\approx 0.014, &
        \|\boprinv\|_{\linop{\statespace}{\controlspace}} &\approx 0.19, &
        \|\A \boprinv\|_{\linop{\statespace}{\controlspace}} &\approx 0.001.
    \end{aligned}
\end{equation*}
The limits are also depicted in \Cref{fig::stokes_params} as dashed purple lines. Since $\An$ is symmetric we conclude $M_n = 1$ for all $n\in\N$ and thus $\lim_{n\to\infty} M_n = 1$. We use these values for the evaluation of the error estimator.

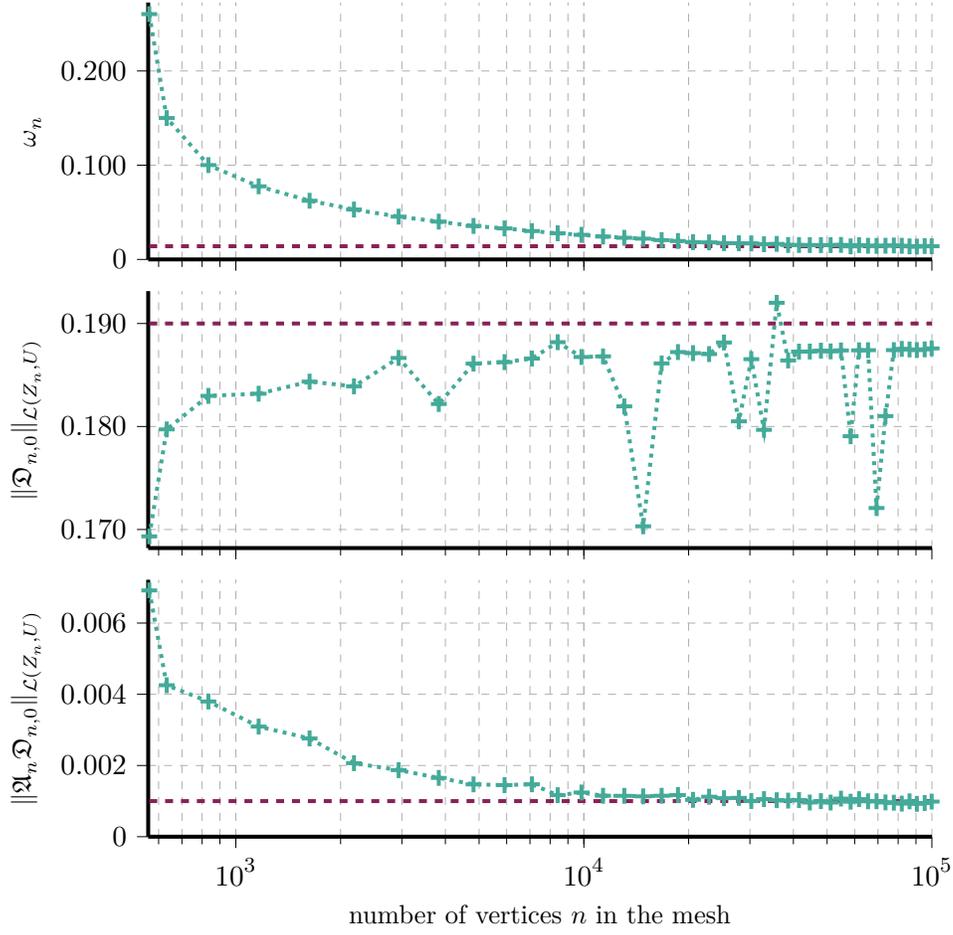
\begin{figure}
    \centering 
    \begin{tikzpicture}

        \definecolor{darkgrey176}{RGB}{176,176,176}

        \begin{groupplot}[
        group style={group size=1 by 3,   
            xticklabels at=edge bottom,
            vertical sep=12pt,
        }, 
        width=0.8\textwidth,
        ]
        \nextgroupplot[
        height=5cm,
        log basis x={10},
        tick align=outside,
        tick pos=left,
        grid=both,
			grid style={line width=.1pt, dashed, draw=gray!10},
			major grid style={line width=.2pt,draw=gray!50},
			axis lines*=left,
			axis line style={line width=\plotlinewidth},
        x grid style={darkgrey176},
        xmin=560, xmax=100000,
        xmode=log,
        xtick style={color=black},
        y grid style={darkgrey176},
        ylabel={\small \(\displaystyle \omega_n\)},
        ymin=0, ymax=0.272695677846172,
        ytick style={color=black},
        ytick={0, 0.1, 0.2},
        yticklabels = {0, 0.100, 0.200},
        ]
        \addplot [line width=\plotlinewidth, color=c_teal, dotted, mark=+, mark size=3, mark options={solid}]
        table {%
        564 0.260358051774104
        633 0.150049655091436
        834 0.100130993734146
        1163 0.0775816940569736
        1629 0.0622897007499808
        2185 0.0529468818933207
        2936 0.0454057453020293
        3830 0.0401163497209935
        4819 0.0355195829487871
        5917 0.0329349398178706
        7092 0.0301221057323283
        8410 0.0277372364254192
        9845 0.0260918661770781
        11355 0.0242723098406486
        13068 0.0229554015443177
        14807 0.0220359098280353
        16730 0.020485381457815
        18640 0.019559180382388
        20593 0.0183038774969163
        22902 0.0180909812854528
        25272 0.0172730841109303
        27889 0.0171705171854216
        30283 0.0170255711760419
        32927 0.0160951791443815
        35824 0.0164683659603743
        38643 0.015331379956323
        41515 0.015288905719471
        44592 0.0152317969671353
        47992 0.0151663109475751
        51176 0.0151655916766899
        54804 0.0153972024315652
        58413 0.0138699096629253
        61781 0.0151846647672484
        65562 0.0145202696856033
        69347 0.0143516625721193
        73581 0.0145102779561559
        77771 0.0147775412085069
        81950 0.0144583716924344
        86259 0.0138762869190927
        90533 0.0136055303327393
        95355 0.0142110366288704
        99851 0.0139789872411679
        };
        \addplot [line width=\plotlinewidth, color=c_wine, dashed]
        table {%
        564 0.014
        100000 0.014
        };

        \nextgroupplot[
            height=5cm,
        log basis x={10},
        tick align=outside,
        tick pos=left,
        x grid style={darkgrey176},
        grid=both,
			grid style={line width=.1pt, dashed, draw=gray!10},
			major grid style={line width=.2pt,draw=gray!50},
			axis lines*=left,
			axis line style={line width=\plotlinewidth},
        xmin=560, xmax=100000,
        xmode=log,
        xtick style={color=black},
        y grid style={darkgrey176},
        ylabel={\small \(\displaystyle \|\mathfrak{D}_{n,0}\|_{\linop{\approxstatespace}{\controlspace}}\)},
        ymin=0.168197722634604, ymax=0.193145475795954,
        ytick style={color=black},
        ytick={0.170, 0.180, 0.190},
        yticklabels = {0.170, 0.180, 0.190},
        ]
        \addplot [line width=\plotlinewidth, color=c_teal, dotted, mark=+, mark size=3, mark options={solid}]
        table {%
        564 0.169331711414665
        633 0.179719773622872
        834 0.18297345223594
        1163 0.183182542969389
        1629 0.184358888431656
        2185 0.183899370971699
        2936 0.186656304856748
        3830 0.182184600160994
        4819 0.186105307262569
        5917 0.186249514458968
        7092 0.186594387330678
        8410 0.1881876692373
        9845 0.186750554805575
        11355 0.186808737176654
        13068 0.181954613071687
        14807 0.170310623410692
        16730 0.186125738255162
        18640 0.187227528855748
        20593 0.187124629660761
        22902 0.187031718858555
        25272 0.188154690270759
        27889 0.180501792386773
        30283 0.186537425242128
        32927 0.179675996116463
        35824 0.192011487015893
        38643 0.186413367579428
        41515 0.187271141586306
        44592 0.18727909936631
        47992 0.187362528574029
        51176 0.187292506485032
        54804 0.187386294656529
        58413 0.179056126938952
        61781 0.18739643748411
        65562 0.187410700349915
        69347 0.172085302261448
        73581 0.181007838875163
        77771 0.187418192124589
        81950 0.187536700476661
        86259 0.187478378622577
        90533 0.187422910141135
        95355 0.187490935812797
        99851 0.187581913696612
        };
        \addplot [line width=\plotlinewidth, color=c_wine, dashed]
        table {%
        564 0.19
        100000 0.19
        };

        \nextgroupplot[
            height=5cm,
            log basis x={10},
            tick align=outside,
            tick pos=left,
            grid=both,
			grid style={line width=.1pt, dashed, draw=gray!10},
			major grid style={line width=.2pt,draw=gray!50},
			axis lines*=left,
			axis line style={line width=\plotlinewidth},
            x grid style={darkgrey176},
            xlabel={\small{number of vertices $n$ in the mesh}},
            xmin=560, xmax=100000,
            xmode=log,
            xtick style={color=black},
            y grid style={darkgrey176},
            ylabel={\small \(\displaystyle \|\mathfrak{A}_n \mathfrak{D}_{n,0}\|_{\linop{\approxstatespace}{\controlspace}}\)},
            ymin=0, ymax=0.00721603899364489,
            ytick style={color=black},
            ytick={0, 0.002, 0.004, 0.006},
            yticklabels = {0, 0.002, 0.004, 0.006},
            scaled ticks = false
        ]
        \addplot [line width=\plotlinewidth, color=c_teal, dotted, mark=+, mark size=3, mark options={solid}]
        table {%
        564 0.00691636961250921
        633 0.00425373308742784
        834 0.00379772141087914
        1163 0.00309253127216058
        1629 0.00275929521677537
        2185 0.00206429498639372
        2936 0.00186948011019869
        3830 0.00165174027756141
        4819 0.00147002946177631
        5917 0.00144835989845069
        7092 0.0014715826909362
        8410 0.0011663428638326
        9845 0.00124163163629679
        11355 0.00115038540175669
        13068 0.00114782862002423
        14807 0.00113398826238778
        16730 0.0011498646729133
        18640 0.00116933431084633
        20593 0.00103507340941228
        22902 0.00112257229058042
        25272 0.00108204248016105
        27889 0.00109029448860967
        30283 0.00100105253755209
        32927 0.0010559275780585
        35824 0.00102544559627561
        38643 0.00101351044027792
        41515 0.00102894267488489
        44592 0.000960373143092605
        47992 0.000999250779634745
        51176 0.000959857642141892
        54804 0.00106976723705943
        58413 0.000982928092875311
        61781 0.00105561448878511
        65562 0.00099482843005233
        69347 0.000996131150925958
        73581 0.000962688530044692
        77771 0.000954967091586838
        81950 0.000929936074651692
        86259 0.000993425847342674
        90533 0.00092298198979563
        95355 0.000946126125655954
        99851 0.000986022047649769
        };
        \addplot [line width=\plotlinewidth, color=c_wine, dashed]
        table {%
        564 0.001
        100000 0.001
        };
        \end{groupplot}

    \end{tikzpicture}
    \caption{Numerically determined values of $\omega, M, \|\A\boprinv\|_{\linop{\controlspace}{\statespace}}$, and $\|\boprinv\|_{\linop{\controlspace}{\statespace}}$ for the Stokes equation on the rectangular domain with a circular hole.}
    \label{fig::stokes_params}
\end{figure}

\subsubsection{\PINN implementation} 
\label{subsub::PINNImplStokes}
We implemented the Stokes flow around the cylinder in a \PINN in the framework published with \cite{HilU25} using harmonic feature embeddings \cite{KasH24} to inform the \PINN of the topology of the considered domain. The core idea of harmonic feature embeddings is that the spatial coordinates are encoded by the Laplace--Beltrami eigenfunctions. Let $\psi_{\NN}$ denote the evaluation of a general neural network and $\phi_{1, ..., n}$ denote the Laplace--Beltrami eigenfunctions. For the two dimensional case at hand, this means that instead of optimizing the \NN hyperparameters such that 
\begin{equation*}
    \psi_{\NN}(x, y, t) 
\end{equation*}
fulfills some data-driven and some physics-informed constraints, we aim for optimizing the {\NN}s hyperparameters such that 
\begin{equation*}
    \psi_{\NN, \abbr{HF}}(x, y, t) \coloneqq \tilde{\psi}_{\NN}(\phi_1(x, y), ..., \phi_n(x, y) , t)
\end{equation*}
fulfills the same data-driven and physics-informed constraints. A suitable network architecture is depicted in \Cref{fig:network_architecture}. It is worth noting, that the second layer of the network (surrounded by the blue box) has no trainable hyperparameters but is completely determined by the domains eigenfunctions. 

For the simplicity of implementation, the neural network used in the provided code does not have the harmonic feature embedding integrated but relies on a preprocessing of the input by the eigenfunctions of the Laplace--Beltrami operator. For proper derivatives w.r.t.~the space coordinates we further provide the neural network during training with the spatial derivatives of the eigenfunctions. Details can be retrieved from the code or from \cite{KasH24}. 

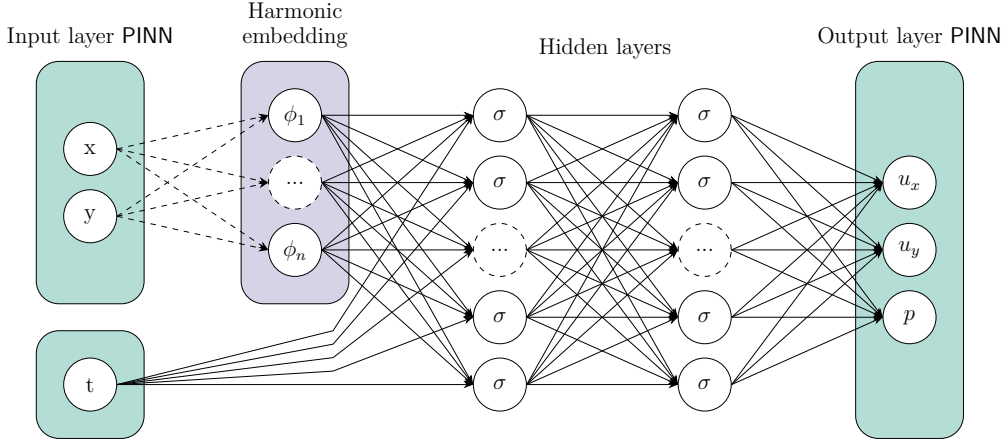
\begin{figure}[!ht]
\centering
\resizebox{0.9\textwidth}{!}{%
    \begin{tikzpicture}
        \tikzstyle{every node}=[font=\large]
        
        \def\xInputLayer{5.5}  
        \def\dxLayer{3.8}  
        
        \def\yCenterNode{12} 
        \def\dyNodes{1.25} 
        \def\dyBoxWidth{1} 

        \def\nodeRadius{28} 

        \draw [ rounded corners = 12.0, fill=c_indigo!20] (\xInputLayer-\dyBoxWidth, \yCenterNode + 2*\dyNodes + \dyBoxWidth) rectangle (\xInputLayer+\dyBoxWidth, \yCenterNode - 0*\dyNodes - \dyBoxWidth);
        \draw [ rounded corners = 12.0, fill=c_teal!40] (\xInputLayer-\dxLayer-\dyBoxWidth, \yCenterNode + 2*\dyNodes + \dyBoxWidth) rectangle (\xInputLayer+\dyBoxWidth-\dxLayer, \yCenterNode - 0*\dyNodes - \dyBoxWidth);
        \draw [ rounded corners = 12.0, fill=c_teal!40] (\xInputLayer- \dxLayer-\dyBoxWidth,\yCenterNode - 2*\dyNodes+\dyBoxWidth) rectangle (\xInputLayer-\dxLayer+\dyBoxWidth,\yCenterNode - 2*\dyNodes-\dyBoxWidth);
        \node [font=\large] at (\xInputLayer -\dxLayer ,\yCenterNode + 3*\dyNodes + 0.2) {Input layer \PINN};
        \node [font=\large] at (\xInputLayer,\yCenterNode + 3*\dyNodes + 0.2) {embedding};
        \node [font=\large] at (\xInputLayer,\yCenterNode + 3*\dyNodes + 0.2 + 0.5) {Harmonic};
        \node[circle, draw, fill=white, minimum size=\nodeRadius] (xnode) at (\xInputLayer-\dxLayer,\yCenterNode+1.5*\dyNodes) {\large x};
        \node[circle, draw, fill=white, minimum size=\nodeRadius] (ynode) at (\xInputLayer-\dxLayer,\yCenterNode+0.5*\dyNodes) {\large y};

        \node[circle, draw, minimum size=\nodeRadius, fill=white] (phi_1) at (\xInputLayer,\yCenterNode + 2*\dyNodes) {$\phi_1$};
        \node[circle, draw, dashed, minimum size=\nodeRadius, fill=white] (phi_3) at (\xInputLayer, \yCenterNode + \dyNodes) {...};
        \node[circle, draw, minimum size=\nodeRadius, fill=white] (phi_5) at (\xInputLayer,\yCenterNode +0*\dyNodes) {$\phi_n$};

        \node[circle, draw, minimum size=\nodeRadius, fill=white] (tnode) at (\xInputLayer-\dxLayer,\yCenterNode - 2*\dyNodes) {t};

        \node (interrupt5) at (\xInputLayer + 0.15*\dxLayer, \yCenterNode-2*\dyNodes) {};
        \node (interrupt4) at (\xInputLayer + 0.15*\dxLayer, \yCenterNode-1.8*\dyNodes) {};
        \node (interrupt3) at (\xInputLayer + 0.15*\dxLayer, \yCenterNode-1.6*\dyNodes) {};
        \node (interrupt2) at (\xInputLayer + 0.15*\dxLayer, \yCenterNode-1.4*\dyNodes) {};
        \node (interrupt1) at (\xInputLayer + 0.15*\dxLayer, \yCenterNode-1.2*\dyNodes) {};

        \node [font=\large] at (11.25,\yCenterNode + 3*\dyNodes) {Hidden layers};

        \node[circle, draw, minimum size=\nodeRadius] (sigma_11) at (\xInputLayer+\dxLayer,\yCenterNode + 2*\dyNodes) {$\sigma$};
        \node[circle, draw, minimum size=\nodeRadius] (sigma_12) at (\xInputLayer+\dxLayer,\yCenterNode + 1*\dyNodes) {$\sigma$};
        \node[circle, draw, dashed, minimum size=\nodeRadius] (sigma_13) at (\xInputLayer+\dxLayer, \yCenterNode - 0*\dyNodes) {...};
        \node[circle, draw, minimum size=\nodeRadius] (sigma_14) at (\xInputLayer+\dxLayer,\yCenterNode - 1*\dyNodes) {$\sigma$};
        \node[circle, draw, minimum size=\nodeRadius] (sigma_15) at (\xInputLayer+\dxLayer,\yCenterNode - 2*\dyNodes) {$\sigma$};

        \node[circle, draw, minimum size=\nodeRadius] (sigma_21) at (\xInputLayer+2*\dxLayer,\yCenterNode + 2*\dyNodes) {$\sigma$};
        \node[circle, draw, minimum size=\nodeRadius] (sigma_22) at (\xInputLayer+2*\dxLayer,\yCenterNode + 1*\dyNodes) {$\sigma$};
        \node[circle, draw, dashed, minimum size=\nodeRadius] (sigma_23) at (\xInputLayer+2*\dxLayer, \yCenterNode - 0*\dyNodes) {...};
        \node[circle, draw, minimum size=\nodeRadius] (sigma_24) at (\xInputLayer+2*\dxLayer,\yCenterNode - 1*\dyNodes) {$\sigma$};
        \node[circle, draw, minimum size=\nodeRadius] (sigma_25) at (\xInputLayer+2*\dxLayer,\yCenterNode - 2*\dyNodes) {$\sigma$};

        \draw [ rounded corners = 12.0, fill=c_teal!40] (\xInputLayer+3*\dxLayer-\dyBoxWidth,\yCenterNode + 2*\dyNodes + \dyBoxWidth) rectangle (\xInputLayer+3*\dxLayer+\dyBoxWidth,\yCenterNode - 2*\dyNodes - \dyBoxWidth);
        \node [font=\large] at (\xInputLayer+3*\dxLayer ,\yCenterNode + 3*\dyNodes + 0.2) {Output layer \PINN};

        \node[circle, draw, minimum size=\nodeRadius, fill=white] (out1) at (\xInputLayer+3*\dxLayer,\yCenterNode + 1*\dyNodes) {$u_x$};
        \node[circle, draw, minimum size=\nodeRadius, fill=white] (out2) at (\xInputLayer+3*\dxLayer, \yCenterNode - 0*\dyNodes) {$u_y$};
        \node[circle, draw, minimum size=\nodeRadius, fill=white] (out3) at (\xInputLayer+3*\dxLayer,\yCenterNode - 1*\dyNodes) {$p$};        

        \draw [->, >=Stealth, dashed] (xnode.east) -- (phi_1.west);
        \draw [->, >=Stealth, dashed] (xnode.east) -- (phi_3.west);
        \draw [->, >=Stealth, dashed] (xnode.east) -- (phi_5.west);

        \draw [->, >=Stealth, dashed] (ynode.east) -- (phi_1.west);
        \draw [->, >=Stealth, dashed] (ynode.east) -- (phi_3.west);
        \draw [->, >=Stealth, dashed] (ynode.east) -- (phi_5.west);

        \draw [->, >=Stealth] (phi_1.east) -- (sigma_11.west);
        \draw [->, >=Stealth] (phi_1.east) -- (sigma_12.west);
        \draw [->, >=Stealth] (phi_1.east) -- (sigma_13.west);
        \draw [->, >=Stealth] (phi_1.east) -- (sigma_14.west);
        \draw [->, >=Stealth] (phi_1.east) -- (sigma_15.west);


        \draw [->, >=Stealth] (phi_3.east) -- (sigma_11.west);
        \draw [->, >=Stealth] (phi_3.east) -- (sigma_12.west);
        \draw [->, >=Stealth] (phi_3.east) -- (sigma_13.west);
        \draw [->, >=Stealth] (phi_3.east) -- (sigma_14.west);
        \draw [->, >=Stealth] (phi_3.east) -- (sigma_15.west);


        \draw [->, >=Stealth] (phi_5.east) -- (sigma_11.west);
        \draw [->, >=Stealth] (phi_5.east) -- (sigma_12.west);
        \draw [->, >=Stealth] (phi_5.east) -- (sigma_13.west);
        \draw [->, >=Stealth] (phi_5.east) -- (sigma_14.west);
        \draw [->, >=Stealth] (phi_5.east) -- (sigma_15.west);

        \draw [->, >=Stealth] (tnode.east) -- (interrupt1.east) -- (sigma_11.west);
        \draw [->, >=Stealth] (tnode.east) -- (interrupt2.east) -- (sigma_12.west);
        \draw [->, >=Stealth] (tnode.east) -- (interrupt3.east) -- (sigma_13.west);
        \draw [->, >=Stealth] (tnode.east) -- (interrupt4.east) -- (sigma_14.west);
        \draw [->, >=Stealth] (tnode.east) -- (interrupt5.east) -- (sigma_15.west);

        \draw [->, >=Stealth] (sigma_11.east) -- (sigma_21.west);
        \draw [->, >=Stealth] (sigma_11.east) -- (sigma_22.west);
        \draw [->, >=Stealth] (sigma_11.east) -- (sigma_23.west);
        \draw [->, >=Stealth] (sigma_11.east) -- (sigma_24.west);
        \draw [->, >=Stealth] (sigma_11.east) -- (sigma_25.west);

        \draw [->, >=Stealth] (sigma_12.east) -- (sigma_21.west);
        \draw [->, >=Stealth] (sigma_12.east) -- (sigma_22.west);
        \draw [->, >=Stealth] (sigma_12.east) -- (sigma_23.west);
        \draw [->, >=Stealth] (sigma_12.east) -- (sigma_24.west);
        \draw [->, >=Stealth] (sigma_12.east) -- (sigma_25.west);

        \draw [->, >=Stealth] (sigma_13.east) -- (sigma_21.west);
        \draw [->, >=Stealth] (sigma_13.east) -- (sigma_22.west);
        \draw [->, >=Stealth] (sigma_13.east) -- (sigma_23.west);
        \draw [->, >=Stealth] (sigma_13.east) -- (sigma_24.west);
        \draw [->, >=Stealth] (sigma_13.east) -- (sigma_25.west);

        \draw [->, >=Stealth] (sigma_14.east) -- (sigma_21.west);
        \draw [->, >=Stealth] (sigma_14.east) -- (sigma_22.west);
        \draw [->, >=Stealth] (sigma_14.east) -- (sigma_23.west);
        \draw [->, >=Stealth] (sigma_14.east) -- (sigma_24.west);
        \draw [->, >=Stealth] (sigma_14.east) -- (sigma_25.west);

        \draw [->, >=Stealth] (sigma_15.east) -- (sigma_21.west);
        \draw [->, >=Stealth] (sigma_15.east) -- (sigma_22.west);
        \draw [->, >=Stealth] (sigma_15.east) -- (sigma_23.west);
        \draw [->, >=Stealth] (sigma_15.east) -- (sigma_24.west);
        \draw [->, >=Stealth] (sigma_15.east) -- (sigma_25.west);

        \draw [->, >=Stealth] (sigma_21.east) -- (out1.west);
        \draw [->, >=Stealth] (sigma_21.east) -- (out2.west);
        \draw [->, >=Stealth] (sigma_21.east) -- (out3.west);

        \draw [->, >=Stealth] (sigma_22.east) -- (out1.west);
        \draw [->, >=Stealth] (sigma_22.east) -- (out2.west);
        \draw [->, >=Stealth] (sigma_22.east) -- (out3.west);

        \draw [->, >=Stealth] (sigma_23.east) -- (out1.west);
        \draw [->, >=Stealth] (sigma_23.east) -- (out2.west);
        \draw [->, >=Stealth] (sigma_23.east) -- (out3.west);

        \draw [->, >=Stealth] (sigma_24.east) -- (out1.west);
        \draw [->, >=Stealth] (sigma_24.east) -- (out2.west);
        \draw [->, >=Stealth] (sigma_24.east) -- (out3.west);

        \draw [->, >=Stealth] (sigma_25.east) -- (out1.west);
        \draw [->, >=Stealth] (sigma_25.east) -- (out2.west);
        \draw [->, >=Stealth] (sigma_25.east) -- (out3.west);
    \end{tikzpicture}
    }%
\caption{\PINN used for the implementation of the Stokes flow around a cylinder. The dashed lines between the spatial variables $x$ and $y$ and the input layer of the \PINN indicate, that this is performed outside of the \PINN implementation.}
\label{fig:network_architecture}
\end{figure}

The \PINNs loss function contains multiple contributions, given by
\begin{subequations}
\begin{align}
    L(\tilde{\vel}, \tilde{p}) &= \frac{\alpha_{\text{evo}}}{N_{\text{evo}}}\sum_{i=1}^{N_{\text{evo}}} \left\| \partial_t \tilde{\vel}(t_i^{\text{evo}}, \mathbf{x}^{\text{evo}}_i) + \nabla \tilde{p}(t^{\text{evo}}_i,\mathbf{x}^{\text{evo}}_i) - \mu \Delta \tilde{\vel}(t_i^{\text{evo}}, \mathbf{x}_i^{\text{evo}})\right\|_2 \label{eq::loss_evo} \\
    &+ \frac{\alpha_{\text{div}}}{N_{\text{div}}} \sum_{i=1}^{N_{\text{div}}} \left|\nabla \cdot \tilde{\vel} (t_i^{\text{div}}, \mathbf{x}_i^{\text{div}})\right| \label{eq::loss_div} \\ 
    &+ \frac{\alpha_{\text{inlet}}}{N_{\text{inlet}}} \sum_{i=1}^{N_{\text{inlet}}} \left\|\tilde{\vel}(t_i^{\text{inlet}}, \mathbf{x_i}^{\text{inlet}}) - \vel_{\text{inlet}}(t_i^{\text{inlet}}, \mathbf{x_i}^{\text{inlet}}) \right\|_2 \label{eq::loss_inlet} \\
    &+ \frac{\alpha_{\text{walls}}}{N_{\text{walls}}} \sum_{i=1}^{N_{\text{walls}}} \left\|\tilde{\vel}(t_i^{\text{walls}}, \mathbf{x}_i^{\text{walls}}) \right\|_2 + \frac{\alpha_{\text{outlet}}}{N_{\text{outlet}}} \sum_{i=1}^{N_{\text{outlet}}} \left\|\partial_{x}\tilde{\vel}(t_i^{\text{outlet}}, \mathbf{x_i}^{\text{outlet}}) \right\|_2 \label{eq::loss_walls_and_outlet} \\
    &+ \frac{\alpha_{\text{init}}}{N_{\text{init}}} \sum_{i=1}^{N_{\text{init}}} \left\|\tilde{\vel}(t_i^{\text{init}}, \mathbf{x_i}^{\text{init}}) \right\|_2. \label{eq::loss_init}
\end{align}
\end{subequations}
Here, we have similarily to \cite{HilU25} contributions from the evolution equation \eqref{eq::loss_evo}, which is characteristic for all neural network in the category \PINN. Furthermore, we have contributions \eqref{eq::loss_div}, \eqref{eq::loss_inlet}, and \eqref{eq::loss_walls_and_outlet} which are soft constraints to reach a \PINN prediction which is in the desired function space. The last term \eqref{eq::loss_init} enforces the initial condition. 

We use the leading 25 Laplace--Beltrami eigenfunctions and construct a neural network with three additional hidden layers with 40 neurons each. We train it using \abbr{Adam} \cite{KinB15} for \num{10000} epochs with a learning rate of $0.01$ using the  weightings
{\small
\begin{align*}
	N_{\mathrm{walls}} &= \num{1200}, &
	N_{\mathrm{outlet}} &= \num{400}, &
	N_{\mathrm{inlet}} &= \num{400}, &
	N_{\mathrm{init}} &= \num{10750}, &
	N_{\mathrm{evo}} &= \num{10649}, &
	N_{\mathrm{div}} &=\num{10649}.\\
	\alpha_{\mathrm{walls}} &= 10, &
	\alpha_{\mathrm{outlet}} &= 0, &
	\alpha_{\mathrm{inlet}} &= 5, &
	\alpha_{\mathrm{init}} &= 0.5, &
	\alpha_{\mathrm{evo}} &= 0.5,  &
	\alpha_{\mathrm{div}} &= 10.
\end{align*}}

\begin{remark}\label{remark::Dirichlet_harmonic_features}
    The treatment of constant Dirichlet boundary conditions is one of the key benefits of the harmonic feature embedding from \cite{KasH24}. Namely, the Dirichlet boundary values remain spatially constant for a fixed point in time, i.e.,
    \begin{equation*}
        \begin{aligned}
            \tilde{\vel}(x, y, t) &= \mathbf{c}(t) \quad &\text{for } \begin{cases} y = 0, \\y = 0.41,\\ (x-c_x)^2 + (y-c_y)^2 = r^2,\end{cases}\\
        \end{aligned}
    \end{equation*}
    where $c_x = c_y = 0.2$ is the center of the obstacle. The time dependency remains, as time is not embedded via the harmonic feature map but rather passed separately as an input to the neural network, alongside the projected eigenfunction values. Therefore, it is still necessary to enforce
    \begin{equation*}
        \mathbf{c}(t) = 0 
    \end{equation*}
    through the loss function.
\end{remark}

\begin{remark}
    In contrast to the problems considered in \cite{KasH24}, the system we investigate here includes time-dependent boundary conditions. As a result, the harmonic feature embeddings must be constructed in a slightly modified way. Specifically, we extend the computational domain beyond the inlet, where the time dependent boundary conditions are prescribed; see \Cref{fig::stokesdomain} (b). This extension enables the \PINN to effectively learn the temporal behavior at the inlet while preserving awareness of the topology of the domain, in particular the presence of the cylindrical obstacle. Nonetheless, for generating the training data for the neural network, the eigenmodes of the the Laplace--Beltrami, computed on the extended domain, are evaluated only at points within the original domain.
\end{remark}

\subsubsection{Numerical results}
\label{subsub::NumericalResultsStokes}

Qualitatively, the trained \PINN exhibits good agreement with the reference solution obtained using the \emph{finite element method} (\FEM) with linear elements, computed via \abbr{DOLFINx}; see \Cref{fig::Stokes_results_in_u}. Nevertheless, the figure also reveals noticeable discrepancies between the predictions of the two methods. We applied the proposed error estimator and compared it to the reference error, defined as the $\Ltwo$-distance to the \abbr{FEM} prediction. Similarly as before, we denote the individual contributions to the error estimator as 
\begin{equation}\label{eq::stokes_contrib}
	\begin{aligned}
		\varepsilon_{\mathrm{init}} &\coloneqq \mexp{\omega t}\|\delta_0\|_{\statespace}, & 
		\varepsilon_{b, \int, 1} &\coloneqq  \int_0^t \mexp{\omega (t-s)}\|\A \boprinv\|_{\linop{\controlspace}{\statespace}}\|\delta_b(s)\|_{\controlspace} \d{s},\\
        \varepsilon_{b, [0]} &\coloneqq \mexp{\omega t}\|\boprinv \delta_b(0)\|_{\statespace}, & 
        \varepsilon_{b, \int, 2} &\coloneqq  \int_0^t \mexp{\omega (t-s)}\|\boprinv\|_{\linop{\controlspace}{\statespace}}\|\dot{\delta}_b(s)\|_{\controlspace} \d{s}, \\
		\varepsilon_{b, [t]} &\coloneqq \|\boprinv\|_{\linop{\controlspace}{\statespace}} \| \delta_b(t)\|_{\controlspace}, & 
		\varepsilon_{\mathrm{evo}} &\coloneqq  \int_0^t \mexp{\omega (t-s)} \|\delta(s)\|_{\statespace}\d{s}. \\
	\end{aligned}
\end{equation}
The results are depicted in \Cref{fig::Stokes_error}, which illustrates the validity of the error estimator and its applicability to problems on complicated domains. In contrast to the previous example, we observe the overall error is dominated by the \PDE residual term.

\begin{figure}
    \centering
    \includegraphics[width=0.48\linewidth]{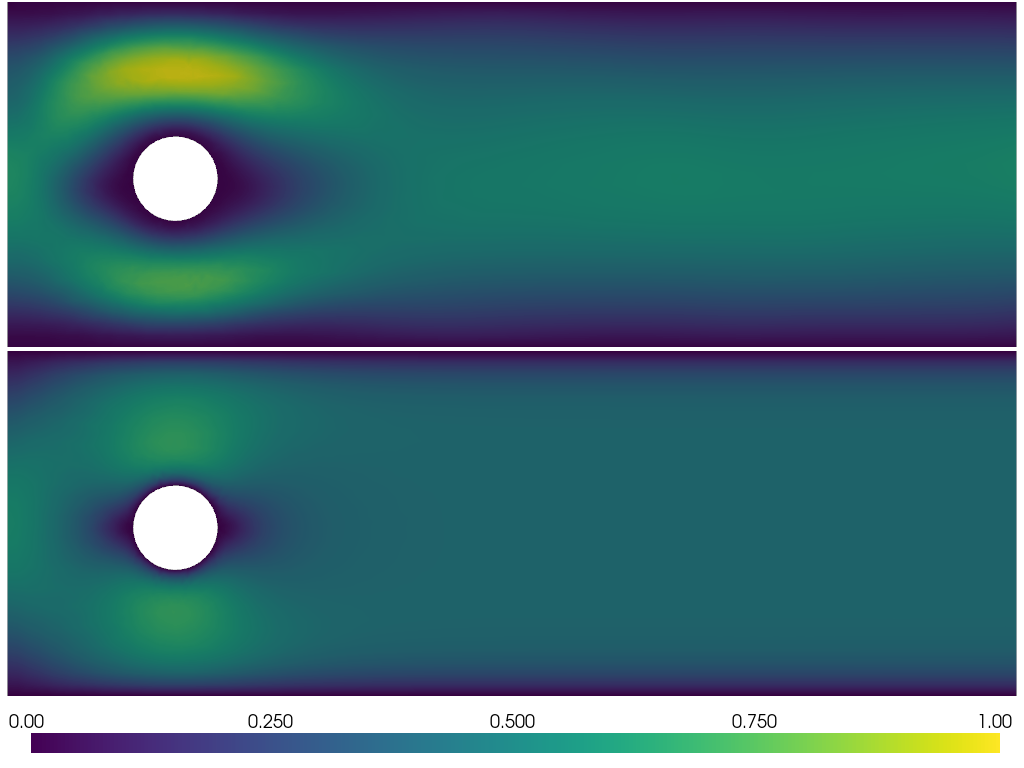}
    \includegraphics[width=0.48\linewidth]{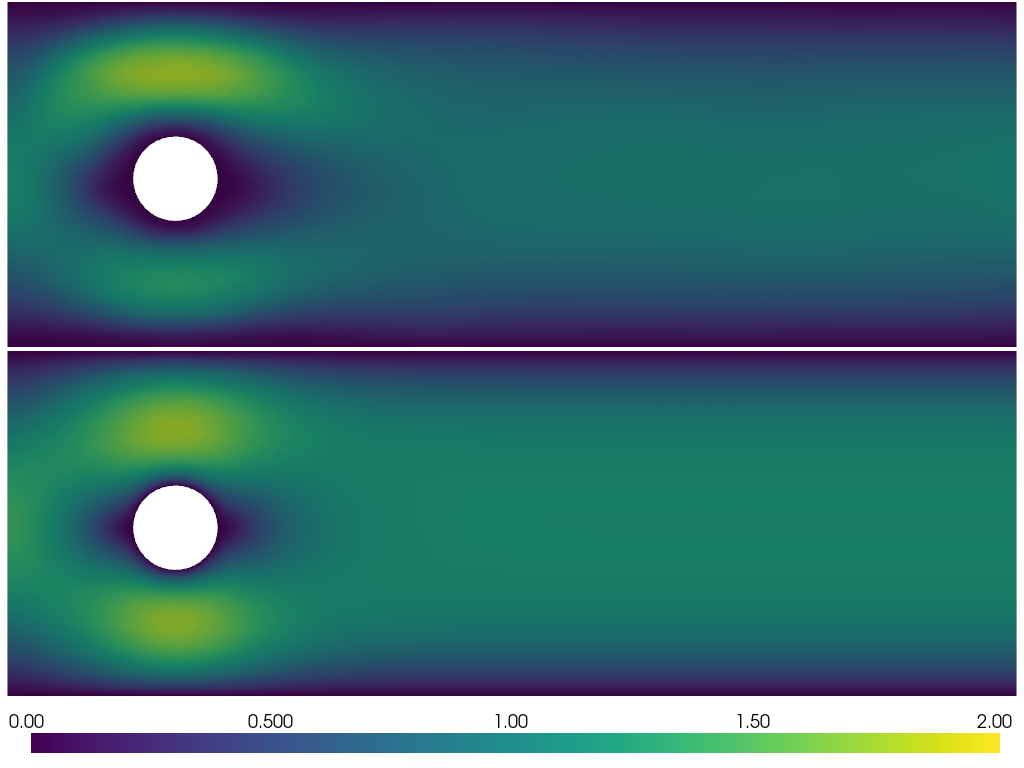}
    \caption{Comparison between the magnitude of the velocity of the \PINN prediction (top figures) with the \abbr{FEM} results (bottom figures) for two points in time $t = 1.0$ (left) and $t = 3.0$ (right). The maximum magnitudes in the colorscale for the different points in time have been chosen suitably for the differences between the two approximations to be visible. }
    \label{fig::Stokes_results_in_u}
\end{figure}

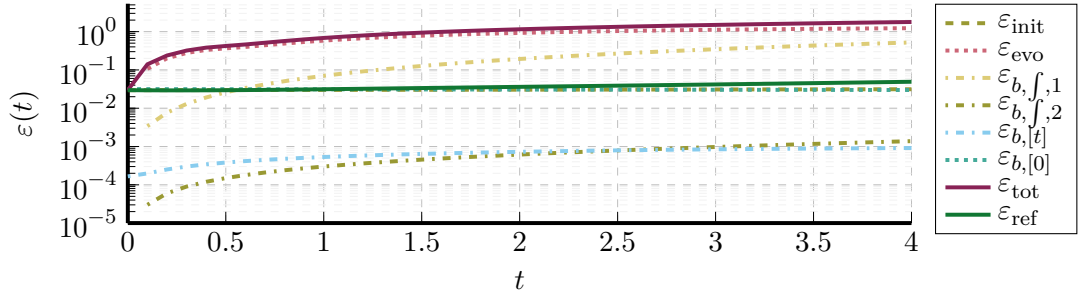
\begin{figure}
    \centering
    \begin{tikzpicture}
        \begin{axis}[
        legend cell align={left},
        legend style={
          fill opacity=0.8,
          draw opacity=1,
          text opacity=1,
          at={(0.97,0.03)},
          anchor=south east,
        },
        width=0.8\linewidth,
        height=0.3\linewidth,
        grid=both,
		grid style={line width=.1pt, dashed, draw=gray!10},
		major grid style={line width=.2pt,draw=gray!50},
		axis lines*=left,
		axis line style={line width=\plotlinewidth},
        xlabel={\(\displaystyle t\)},
        xmajorgrids,
        xmin=0, xmax=4,
        xminorgrids,
        xtick style={color=black},
        ylabel={\(\displaystyle \varepsilon(t)\)},
        ymajorgrids, 
        ymode=log,
        ytick style={color=black},
        ytick={1e-05,0.0001,0.001,0.01,0.1,1,10,100,1000},
        legend pos=outer north east
        ]
        \addplot [line width=\plotlinewidth, c_olive, dashed]
        table[x=t,y=E_init,col sep=comma]{figures/stokes_flow/error_contributions.csv};
        \addlegendentry{\(\displaystyle \varepsilon_{\mathrm{init}}\)}
        
        \addplot [line width=\plotlinewidth, c_rose, dotted]
        table[x=t,y=E_PI,col sep=comma]{figures/stokes_flow/error_contributions.csv};
        \addlegendentry{\(\displaystyle \varepsilon_{\mathrm{evo}}\)}
        
        \addplot [line width=\plotlinewidth, c_sand, dash pattern=on 1pt off 3pt on 3pt off 3pt]
        table[x=t,y=E_bc_int_ub,col sep=comma]{figures/stokes_flow/error_contributions.csv};
        \addlegendentry{\(\displaystyle \varepsilon_{b, \int, 1}\)}
        
        \addplot [line width=\plotlinewidth, c_olive, dash pattern=on 1pt off 3pt on 3pt off 3pt]
        table[x=t,y=E_bc_int_ubdot,col sep=comma]{figures/stokes_flow/error_contributions.csv};
        \addlegendentry{\(\displaystyle \varepsilon_{b, \int, 2}\)}
        
        \addplot [line width=\plotlinewidth, c_cyan, dash pattern=on 1pt off 3pt on 3pt off 3pt]
        table[x=t,y=E_bc_sum_ubt,col sep=comma]{figures/stokes_flow/error_contributions.csv};
        \addlegendentry{\(\displaystyle \varepsilon_{b, [t]}\)}
        
        \addplot [line width=\plotlinewidth, c_teal, dotted]
        table[x=t,y=E_bc_sum_ub0,col sep=comma]{figures/stokes_flow/error_contributions.csv};
        \addlegendentry{\(\displaystyle \varepsilon_{b, [0]}\)}
        
        \addplot [line width=\plotlinewidth, c_wine]
        table[x=t,y=E_tot,col sep=comma]{figures/stokes_flow/error_contributions.csv};
        \addlegendentry{\(\displaystyle \varepsilon_\mathrm{tot}\)}

        \addplot [line width=\plotlinewidth, c_green]
        table[x=t,y=E_ref,col sep=comma]{figures/stokes_flow/error_contributions.csv};
        \addlegendentry{\(\displaystyle \varepsilon_\mathrm{ref}\)}
        \end{axis}
    \end{tikzpicture}
    \caption{Different contributions~\eqref{eq::stokes_contrib} to the error estimator and their sum $\varepsilon_{\mathrm{tot}}$ in comparison with the reference error $\varepsilon_{\mathrm{ref}}$, which was computed as the $\Ltwo$-error between the \PINN prediction and the \abbr{FEM} reference solution.}
    \label{fig::Stokes_error}
\end{figure}

\section{Discussion}

In this contribution, we modified the previously introduced error estimator for \PINNs from \cite{HilU25} to extend its applicability to a broader class of problems. Unlike the original version, the new error estimator does not rely on the theory of input-to-state stability and can thus be applied to systems that are not necessarily stable. Additionally, we developed the necessary theory to estimate the constants in the error estimator using conventional numerical approximation methods, such that is easier accessible to practitioners. We illustrated both benefits with an academic toy example (for the numerical approximation of the key parameters) and the Stokes equation on a two-dimensional domain with a circular obstacle.

The numerical studies confirm the theoretical results, i.e., the error estimator is indeed a true upper bound on the prediction error with a reasonable overestimation of the true error in comparison to similar results from the literature. Furthermore, the numerical experiments showed that the temporal derivative of the boundary error can be weighted similarily to the boundary loss or even omitted entirely. This is in agreement with the theoretical expectation from the full Fattorini trick, which suggests that whenever the system is admissible, the temporal derivative of the boundary error can be absorbed in an alternative operator which captures the impact of the boundary error on the state in the interior of the domain. Fully exploiting this observation in combination with a numerical approximation of the operator is highly non-trivial and remains open for future research. 

Finally, our analysis is restricted to boundary operators that admit linear, bounded right inverses. However, the theory of nonlinear right inverses is well-established in functional analysis. Extending our results to include such operators represents another promising avenue for further investigation.

\section*{Statements and Declarations}
\subsubsection*{Acknowledgements}
Both authors acknowledge funding by the Deutsche Forschungsgemeinschaft (DFG, German Research Foundation) – Project-ID 258734477 – SFB 1173. BH additionally acknowledges funding from the DFG under Germany's Excellence Strategy -- EXC 2075 -- 390740016 and from the International Max Planck Research School for Intelligent Systems (IMPRS-IS). Major parts of this manuscript were written while both authors were affiliated with the University of Stuttgart. 

\subsubsection*{Author contributions}
BH: Conceptualization, methodology, formal analysis, software, investigation, data curation, writing - original draft; BU: Supervision, formal analysis, writing - original draft, funding acquisition.
\bibliographystyle{plain-doi} 
\bibliography{journalAbbr, journalAbbr_other, literature}

\appendix

\section{Derivation of the boundary operator matrices}
\label{app::boundarymat}

We here formalize the step of deriving $\bopn$ and $\bopnrinv$ for the Stokes example from \Cref{subsec:stokes}. For ease of notation and derivation, we derive the boundary operator matrices for linear finite elements and only for a 1D problem, i.e., we only consider $\vel_x$. We assume that the vector entries of the discretized state $\tilde{\vel}_{n}$ firstly contain $\vel_x$ for all nodes and then $\vel_y$ for all nodes. Then, the full boundary operator matrix can be constructed by suitably stacking and concatenating the contributions multiple times. 

For this, we first fix the convention that the $i$th entry of $\tilde{\vel}_{n}$ is the function value of $\vel_x$ at the node of index $i$ at position $(x_i, y_i)$, i.e.,
\begin{equation*}
    \tilde{\vel}_{n} = \left[
        \vel_x(x_1, y_1), \; \dots, \; \vel_x(x_n, y_n), \; \vel_y(x_1, y_1), \, \dots, \; \vel_y(x_n, y_n)
        \right]^\T
    \end{equation*}
here still having $2n$ entries, whereas we ignore the latter $n$ as previously mentioned. All boundary nodes are then collected in
\begin{equation*}
    \mathcal{V}_n \coloneqq \left\lbrace 1 \le k \le n \mid (x_k, y_k) \in \partial\Omega \right\rbrace.
\end{equation*}
For later separation into Neumann and Dirichlet boundary conditions, we further introduce the node sets 
\begin{equation*}
    \begin{aligned}
        \mathcal{V}^D_n \coloneqq \left\lbrace 1 \le k \le n \mid (x_k, y_k) \in \partial\Omega_D \right\rbrace, \quad
        \mathcal{V}^N_n \coloneqq \left\lbrace 1 \le k \le n \mid (x_k, y_k) \in \partial\Omega_N \right\rbrace, 
    \end{aligned}
\end{equation*}
for nodes in the Dirichlet ($\partial\Omega_D$) or Neumann boundaries ($\partial\Omega_D$), respectively. To facilitate independent indexing of boundary values, we introduce the following mapping for all boundary vertices
\begin{equation*}
    \begin{aligned}
        \iota_{\partial\Omega} \colon \{1, \ldots , |\mathcal{V}_n|\} \to \mathcal{V}_n \subsetneq \{1, \ldots , n\}
    \end{aligned}
\end{equation*}
and $\iota_{\partial\Omega}^D, \iota_{\partial\Omega}^N$ for Dirichlet and Neumann boundary vertices, respectively. It maps a dedicated, consecutive boundary index to the general node indexing of the full mesh. For Neumann boundary conditions we need to be able to find the boundary normal and hence introduce the additional mapping 
\begin{equation*}
    \gamma_{\partial\Omega}^N \, \colon \{1, \ldots, |\mathcal{V}_n^N|\} \to \mathcal{V}_n^N \times \{1,\ldots, n\}^2
\end{equation*}
which maps from the boundary node index to the three node indices in the mesh which correspond to the corners of the triangle which contain the surface normal in the sense as depicted in \Cref{fig::neumannboundary}. 

With these preparations in place, we can easily write down the boundary operator matrices for Dirichlet and Neumann boundaries independent of each other. For Dirichlet boundary conditions given by $\vel_{b, 1}(x, y, t)$, this then reads
\begin{equation*}
    \begin{aligned}
        \bopn^D \tilde{\vel}_{1, n} &= \vel_{b, 1, n}, &
        {\bopn^D}_{; i,j} &= \begin{cases}
            1, & \quad \text{if } \iota_{\partial\Omega}^D(i) = j,\\
            0, & \quad \text{else.}
        \end{cases}
    \end{aligned}
\end{equation*}
Further, we find for the boundary operator for the Neumann boundary conditions on the outlet (normal derivative in negative x-direction)
\begin{equation}\label{eq::Neumann_boundary}
    \begin{aligned}
    {\bopn^N}_{;\; i, j} &= \begin{cases}
        \frac{(y_\beta - y_\zeta)}{(x_\alpha - x_\beta)(y_\beta-y_\zeta) - (x_\beta - x_\zeta)(y_\alpha - y_\beta)}, & \quad \text{if } \iota_{\partial\Omega}^N (i) = j, \\ 
        \frac{-y_\alpha + 2 y_\beta - y_\gamma}{(x_\alpha - x_\beta)(y_\beta-y_\zeta) - (x_\beta - x_\zeta)(y_\alpha - y_\beta)}, & \quad \text{if } \gamma_{\partial\Omega}^N (i)[2] = j, \\
        \frac{(y_\alpha - y_\beta)}{(x_\alpha - x_\beta)(y_\beta-y_\zeta) - (x_\beta - x_\zeta)(y_\alpha - y_\beta)}, & \quad \text{if } \gamma_{\partial\Omega}^N (i)[3] = j, \\ 
        0, & \quad \text{else.}
    \end{cases}
    \end{aligned}
\end{equation}
with $\alpha \coloneqq \gamma_{\partial\Omega}^N (i)[1] = \iota_{\partial\Delta}^N (i)$, $\beta \coloneqq \gamma_{\partial\Omega}^N (i)[2]$, and $\zeta \coloneqq \gamma_{\partial\Omega}^N (i)[3]$ for each fixed $i$.
\begin{remark}
    This formula for the Neumann boundary operator matrix come from assuming the linear interpolation between the three points of the triangle under consideration $\phi(x, y) = a x + by + c$ for which in particular the derivative w.r.t.~$x$ is given by $a$. Resolving with known node values and coordinates for $a$ yields the boundary operator \eqref{eq::Neumann_boundary}. 
\end{remark}

The full boundary operator matrix can then be determined by stacking the matrices $\bopn^D$ and $\bopn^N$ and adjusting the right hand side of the boundary equation accordingly. The right inverse $\bopnrinv$ can then be determined easily.

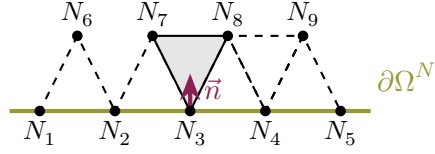
\begin{figure}
	\centering
    \begin{tikzpicture}[scale=1]

    \coordinate (left) at (-0.4,0);
    \coordinate (N1) at (0,0);
    \coordinate (N2) at (1,0);
    \coordinate (N3) at (2,0);
    \coordinate (N4) at (3,0);
    \coordinate (N5) at (4,0);
    \coordinate (right) at (4.4,0);

    \coordinate (N6) at (0.5,1);
    \coordinate (N7) at (1.5,1);
    \coordinate (N8) at (2.5,1);
    \coordinate (N9) at (3.5,1);

    \draw[thick, dashed] (N1) -- (N2) -- (N6) -- cycle;
    \draw[thick, dashed] (N2) -- (N3) -- (N7) -- cycle;
    \draw[thick, fill=gray!20] (N3) -- (N8) -- (N7) -- cycle;
    \draw[thick, dashed] (N3) -- (N4) -- (N8) -- cycle;    
    \draw[thick, dashed] (N4) -- (N5) -- (N9) -- cycle;
    \draw[thick, dashed] (N4) -- (N8) -- (N9) -- cycle;

    \draw[ultra thick, ->, >=Stealth, c_wine] (N3) -- (2, 0.5);
    \node[c_wine] at (2.3, 0.3) {$\vec{n}$};

    \draw[ultra thick, color=c_olive] (left) -- (right) node[pos = 1.1, above=3pt, c_olive] {$\partial\Omega^N$};

    \foreach \i in {1,...,5} {
        \fill (N\i) circle (2pt);
        \node[below] at (N\i) {\small $N_{\i}$};
    }
    \foreach \i in {6,...,9} {
        \fill (N\i) circle (2pt);
        \node[above] at (N\i) {\small $N_{\i}$};
    }

    \end{tikzpicture}
    \caption{Excerpt of a triangular mesh for which the normal vector at the node $N_3$ is depicted. The cell marked in light gray containing $N_3+\epsilon\vec{n}$ is the one considered for determining~$\bopn^N$.}
    \label{fig::neumannboundary}
\end{figure}

\end{document}